\documentclass[11pt]{article}
\usepackage{psfrag,amsmath}
\usepackage{amsthm}
\usepackage{amssymb}
\usepackage{mathrsfs}
\usepackage{enumerate}
\usepackage{amssymb,a4wide,latexsym,makeidx,epsfig,fleqn}
\usepackage{bm}
\usepackage{floatrow}
\usepackage{makecell}
\usepackage{geometry}
\usepackage{enumerate}
\usepackage{graphicx}
\usepackage{multirow}
\usepackage{color}
\usepackage{array}
\usepackage{subfigure}
\usepackage{float}
\usepackage{cases}
\usepackage{comment}
\usepackage[utf8]{inputenc}
\usepackage[T1]{fontenc}

\usepackage[titletoc]{appendix}

\usepackage{longtable}

\allowdisplaybreaks

\usepackage{titlesec}
\titleformat{\chapter}[display]
{\normalfont\Large\bfseries}{\thechapter}{11pt}{\Large}
\titleformat{\section}
{\normalfont\large\bfseries}{\thesection}{11pt}{\large}
\titlespacing*{\chapter}{0pt}{0pt}{15pt} %left, beforesep, aftersep, right
\titlespacing*{\section}{0pt}{3.5ex plus 1ex minus .2ex}{2.3ex plus .2ex}
\usepackage[titletoc]{appendix}

\usepackage{indentfirst}
\newtheorem{theorem}{Theorem}[section]

\newtheorem{lemma}[theorem]{Lemma}

\begin{document}
\textwidth 150mm \textheight 225mm
\title{The $A_\alpha$ spectral radius of $k$-connected graphs with given diameter }
	%\thanks{Supported by the National Natural Science Foundation of China (No. 12271439).} }
\author{{Xichan Liu\textsuperscript{a,b}, Ligong Wang\textsuperscript{a,b,}\footnote{Corresponding author.}}\\
		{\small \textsuperscript{a} School of Mathematics and Statistics
		}\\
		{\small Northwestern Polytechnical University, Xi'an, Shaanxi 710129, P.R. China.}\\
		{\small \textsuperscript{b} Xi'an-Budapest Joint Research Center for Combinatorics}\\
		{\small  Northwestern Polytechnical University, Xi'an, Shaanxi 710129, P.R. China.}\\
		{\small E-mail: xichanliumath@163.com, lgwangmath@163.com}}
\date{}
\maketitle

\begin{center}
\begin{minipage}{135mm}
\vskip 0.3cm
\begin{center}
{\small {\bf Abstract}}
\end{center}
{\small
Let $G$ be a graph with adjacency matrix $A(G)$  and degree diagonal matrix $D (G)$. In 2017, Nikiforov defined the matrix $A_\alpha(G) = \alpha D(G) + (1-\alpha)A(G)$ for any real $\alpha\in[0,1]$. The largest eigenvalue of $A_\alpha(G)$ is called the $A_\alpha$ spectral radius or the $A_\alpha$-index of  $G$.
 Let $\mathcal{G}_{n,k}^d$ be the set of $k$-connected graphs with order $n$ and diameter $d$. In this paper, we determine the graphs with maximum $A_\alpha$ spectral radius among all graphs in $\mathcal{G}_{n,k}^d$ for any $\alpha\in[0,1)$, where $k\geq2$ and $d\geq2$. We generalizes the results of adjacency matrix in [P. Huang, W.C. Shiu, P.K. Sun, Linear Algebra Appl., 2016, Theorem 3.6] and the results of signless Laplacian matrix in [P. Huang, J.X. Li, W.C. Shiu, Linear Algebra Appl., 2021, Theorem 3.4]. %Furthermore, we also obtaine some bounds of the set of graphs $\mathcal{G}_{n,k}^d$.
\vskip 0.1in \noindent {\bf Key Words}:  $A_\alpha$ spectral radius; $k$-connected; diameter  \vskip
0.1in \noindent {\bf AMS Subject Classification (2020)}: \ 05C50}
\end{minipage}
\end{center}

\section{Introduction}\label{intro}
%%XL: ���²���Ҫ
%\textwidth 150mm \textheight 225mm
%\par
All graphs considered here are simple. Let $G$ be a graph with vertex set $V(G) = \{v_1, v_2,\ldots, v_n\}$ and edge set $E(G)$.
The order of $G$ is $n$ and the size of $G$ is $m$, where $n = |V(G)|$ and $m = |E(G)|$.
The neighbor set of a vertex $v$ in $G$ is denoted by $N_G(v)$, the degree of the vertex $v$ in $G$ is denoted by $d_G(v)$, where $d_G(v)=|N_G(v)|$. The minimum degree and the maximum degree of $G$ are denoted by $\delta(G)$ and $\Delta(G)$, respectively. The distance $d_G(u,v)$  between two distinct vertices $u, v $ of a connected graph $G$ is the length of the shortest path connecting them. The diameter $d$ of $G$ is the maximum distance among all distinct vertices pairs of $G$.
For a subset $W\subseteq V(G)$,	let $G[W]$ be the subgraph induced by $W$ and $G-W=G[V(G)\setminus W]$.
A graph is \emph{$k$-connected} if $G-W$ is connected for any subset $W\subseteq V(G)$ with $|W|<k$.
The sequential join $G_1\vee \cdots \vee G_k$ of graphs $G_1, \ldots, G_k$, is the graph formed by taking one copy of each graph and adding additional edges from each vertex of $G_i$ to all vertices of $G_{i+1}$, for $i=1,2,\ldots,k-1$.
A path and a complete graph of order $n$ are denoted by $P_n$ and $K_n$, respectively.
Unless otherwise stated, we use the standard notations and terminologies in \cite{Bondy-b, stevan-b}.

The $M$ spectral radius of a graph $G$ is the largest eigenvalue of $M$, where $M$ is a responding graph matrix defined in a prescribed way, such as adjacency matrix, (signless) Laplacian matrix and others. The adjacency matrix of a graph $G$ with order $n$ is an $n\times n$ 0-1 matrix, denoted by $A(G)=[a_{ij}]_{n\times n}$,
where $a_{ij}=1$ if $v_iv_j\in E(G)$, and $a_{ij}=0$ otherwise.
The degree diagonal matrix of $G$ is $D(G)=diag(d_{G}(v_1), d_{G}(v_2),\ldots, d_{G}(v_n))$. The Laplacian matrix of a graph $G$ is $L(G)=D(G)-A(G)$, %while
and
the signless Laplacian matrix of $G$ is $Q(G)=D(G)+A(G)$. In 2017,
Nikiforov \cite{Nik-1} defined the $A_\alpha$ matrix of a graph $G$ as $A_\alpha(G)= \alpha D(G) + (1-\alpha)A(G)$ for any real $\alpha\in[0,1]$,  which can underpin a unified theory of $A (G)$ and $Q (G)$.
The $A_\alpha$ spectral radius of a graph $G$ is denoted as $\lambda_\alpha(G)$. Based on the well-known Perron-Frobenius theorem, there exists a positive eigenvector
$X=(x_{v_1},x_{v_2},\ldots,x_{v_n})^T$
corresponding to $\lambda_{\alpha}(G)$, also called the Perron vector of $G$. For a vertex $v\in V(G)$, the eigenequation of $A_\alpha(G)$ corresponding to $v$ is written as
\begin{align}\label{AX}
\lambda_\alpha(G)x_v=\alpha d_G(v)x_v+(1-\alpha)\sum_{uv\in E(G)}x_u.
\end{align}
Moreover, we have
\begin{align}\label{xax}
X^TA_\alpha X=\alpha\sum_{u\in V(G)}d_G(u)x_u^2+2(1-\alpha)\sum_{uv\in E(G)}x_ux_v.
\end{align}
By Rayleigh Quotient, we have
\begin{align}\label{rayleigh}
\lambda_{\alpha}(G)=\sup_{ \Vert X \Vert\neq 0}\frac{X^TA_\alpha(G) X}{X^TX}.
\end{align}

The Brualdi-Solheid problem, which was first presented in 1986 by Brualdi and Solheid in article \cite{Brualdi-1}, is a well-known question about finding a tight bound for the spectral radius in a set of graphs and characterizing the extremal graphs.
Many recent results on this problem for various kinds of graphs and their adjacency spectral radius have been obtained, we refer to articles \cite{B-Z-1, F-Y-Z-1, FZL, WG, ZWQ}, some monographs \cite{15stanic, stevan-b} and the references therein. The results of the Brualdi-Solheid problem about the signless Laplacian spectral radius and $A_\alpha$ spectral radius, we refer to articles \cite{GHS,GRS,JLW,LGW,YFI} and \cite{CPC,FW,HLZ,LQ,P,ZHW}, respectively. Further, other results for the $A_\alpha$ spectral radius of digraph can be refered to \cite{Aa-3, Aa-2, Aa-1}  and the references therein.  Recently, the spectral extremal problem of adjacency matrix and $A_\alpha$ matrix has also attracted lots of attention, some results can be found in articles \cite{C-D-T-1,CFT} and the references therein.

Let $\mathcal{G}_{n,k}^d$ be the set of $k$-connected
graphs of order $n$ with given diameter $d$.
Huang, Shiu, Sun \cite{Huang-2} and Huang, Li, Shiu \cite{Huang-1} solved the Brualdi-Solheid problem in $\mathcal{G}_{n,k}^d$ for adjacency spectral radius and signless Laplacian spectral radius, respectively. Their conclusions are as follows.
\begin{theorem}[\cite{Huang-1,Huang-2}]\label{H-main}
For $k\geq 2$ and $d\geq 2$, the graph $K_1\vee K_{n_1}\vee \cdots \vee K_{n_{d-1}}\vee K_1$ attains the maximum $($signless Laplacian$)$ spectral radius in $\mathcal{G}_{n,k}^d$, where $n_i=k$ for $i\in\{1,2,\ldots,d-1\}\setminus\{\lfloor\frac{d}{2}\rfloor\}$, and $n_{\lfloor\frac{d}{2}\rfloor}\geq k$.
\end{theorem}

Furthermore, Huang, Li and Shiu \cite{Huang-1} believed
that the results of Theorem \ref{H-main} also hold for
$A_\alpha$ spectral radius in $\mathcal{G}_{n,k}^d$ with $\alpha\in[0,1)$. For $d=1$, it is obvious that $K_n$ is the unique graph with the maximum $A_\alpha$ spectral radius.
For $k=1$ and $d\geq 2$, Xue et al. \cite{Xuejie-1} characterized that $K_{n-d}(\lfloor\frac{d}{2}\rfloor,\lceil\frac{d}{2}\rceil)$ is the unique graph with the maximum $A_\alpha$ spectral radius, where $K_{n-d}(\lfloor\frac{d}{2}\rfloor,\lceil\frac{d}{2}\rceil)$ is the graph obtained from $K_{n-d}$ by connecting its all vertices to an end vertex of $P_{\lfloor\frac{d}{2}\rfloor}$ and an end vertex of $P_{\lceil\frac{d}{2}\rceil}$. In this paper, we confirm Huang, Li and Shiu's conjecture for all $d\geq 2$ and $k\geq 2$. Our conclusion is as follows.
\begin{theorem}\label{th-main}
For $k\geq 2$ and $d\geq 2$, the graph $G=K_1\vee K_{n_1}\vee \cdots \vee K_{n_{d-1}}\vee K_1$ attains the maximum $A_\alpha$ spectral radius in $\mathcal{G}_{n,k}^d$, where $n_i=k$ for each $i\in\{1,2,\ldots,d-1\}\setminus\{\lfloor\frac{d}{2}\rfloor\}$, $n_{\lfloor\frac{d}{2}\rfloor}\geq 2k$. Further, $\frac{1}{2}\big(\alpha(n_{\lfloor\frac{d}{2}\rfloor}+2k)+\sqrt{\alpha^{2}(n_{\lfloor\frac{d}{2}\rfloor}+2k+1)^2+4(n_{\lfloor\frac{d}{2}\rfloor}+2k)(1-2\alpha)}\big)<\lambda_{\alpha}(G)$$<n_{\lfloor\frac{d}{2}\rfloor}+2k-1.$
\end{theorem}

The rest of this paper is structured as follows. We provide several useful Lemmas in Section \ref{pre} and present our proof of Theorem \ref{th-main} in Section \ref{main}.

\section{ Preliminaries}\label{pre}
In this section, we present some preliminary results, which will be used in Section \ref{main}.

A graph is called \emph{diameter critical} if its diameter decreases with any addition of an edge. The structure of a diameter critical $k$-connected graph was characterized by Ore \cite{Ore-1} in 1968.

\begin{lemma}[\cite{Ore-1}]\label{d-critical}
	Let $G$ be a graph in $\mathcal{G}_{n,k}^d$. If $G$ is a diameter critical graph, then $G\cong K_1\vee K_{n_1}\vee \cdots \vee K_{n_{d-1}}\vee K_1$, where $n_i\geq k$ for each $i=1,2,\ldots,d-1$.
\end{lemma}
Following that, we present several results on the $A_\alpha$ spectral radius. In this paper, the vertices $u$ and $v$ of $G$ are said to be \textit{equivalent}, if there exists an automorphism $p$: $G\rightarrow G$ such that $p(u)=v$.

\begin{lemma}[\cite{Nik-1}]\label{Nik-complete}
	Let $G\cong K_n$. Then $\lambda_{\alpha}(G)=n-1$.
\end{lemma}
\noindent\begin{lemma}[\cite{Nik-1}]\label{le:g-uv} Let $G$ be a connected graph and $G_0$ be a proper subgraph of $G$. Then $\lambda_\alpha(G_0)<\lambda_{\alpha}(G)$ for $\alpha\in[0, 1)$.
\end{lemma}
\noindent\begin{lemma}[\cite{Nik-1}]\label{le:g-u=g-v}
	Let $G$ be a connected graph. Let $u$ and $v$ be two equivalent vertices in $G$. If $X$ is the positive eigenvector of $A_\alpha(G)$ corresponding to $\lambda_\alpha(G)$, then $x_u=x_v$ for $\alpha\in[0,1]$.
\end{lemma}
\noindent\begin{lemma}[\cite{Nik-1}]\label{le:bound1}
	Let $G$ be a graph with maximum degree $\Delta( G)=\Delta$. If $\alpha\in[0,1)$, then $$\lambda_{\alpha}(G)\geq \frac{1}{2}\big(\alpha(\Delta+1)+\sqrt{\alpha^{2}(\Delta+1)^2+4\Delta(1-2\alpha)}\ \big).$$
	If $G$ is connected, then equality holds if and only if $G\cong K_{1,\Delta}$. In particular,
	\begin{equation*}\label{eq:1.1}
	\lambda_\alpha(G)\geq
	\left\{
	\begin{array}{ll}
	\alpha(\Delta+1),&  \textrm{if $\alpha\in [0,\frac{1}{2}],$}\\
	\alpha\Delta+\frac{(1-\alpha)^2}{\alpha},&  \textrm{if $\alpha\in [\frac{1}{2},1).$}\\
	\end{array}
	\right.
	\end{equation*}
\end{lemma}
\noindent\begin{lemma}\label{le:bound2} \rm(\cite{Nik-1})
	Let $G$ be a graph. If $\alpha\in [0,1]$, then
	$$
	\frac{2|E(G)|}{|V(G)|}\leq \lambda_\alpha(G)\leq \max_{uv\in E(G)}\{\alpha d_G(u)+(1-\alpha)d_G(v)\}.
	$$
	The first equality holds if and only if $G$ is regular.
\end{lemma}
\noindent\begin{lemma}[\cite{Xuejie-1}]\label{le:g-uw+vw}
	Let $G$ be a connected graph and $X(G)=(x_1,x_2,\ldots,x_n)^T$ be a positive eigenvector of $A_{\alpha}(G)$ corresponding to $\lambda_{\alpha}(G)$. For two distinct vertices $u$, $v$ of $G$, suppose $Y\subseteq N_G(u)\setminus(N_G(v)\cup\{v\})$. Let $G^{'}=G-\{uw:w\in Y\}+\{vw:w\in Y\}$. If $Y\neq\emptyset$ and $x_v\geq x_u$, then $\lambda_{\alpha}(G^{'})>\lambda_{\alpha}(G)$ for $\alpha\in [0,1)$.
\end{lemma}
In order to prove our main results, we generalize Lemma \ref{le:g-uw+vw}, our conclusion is shown in Lemma \ref{v-u}. In addition, we characterize the $A_\alpha$ spectral radius about the sequential join $K_{n_1}\vee K_{n_2}\vee K_{n_3}$ of three complete graphs $K_{n_1}$, $K_{n_2}$, and $K_{n_3}$, the results are shown in Lemma \ref{bound}.
\begin{lemma}\label{v-u}
	Let $G$ be a connected graph and  $X(G)=(x_1,x_2,\ldots,x_n)^T$ be a positive eigenvector of $A_{\alpha}(G)$ corresponding to $\lambda_{\alpha}(G)$. Let $U$ and $V$ be two disjoint subsets of $V(G)$  that satisfy $|U|=|V|=k$. For each pair of vertices $u_1$, $u_2\in U$ $($resp. $v_1$, $v_2\in V$$)$ are equivalent in $G$, then $x_{u_1}=x_{u_2}$ $($resp. $x_{v_1}=x_{v_2}$$)$. Suppose $W=N_G(U)\setminus \{N_G(V)\cup V\}$. Let $G'$ be a graph from $G$ by deleting all edges between $U$ and $W$ and adding all edges between $V$ and $W$. If $x_{v_1}\geq x_{u_1}$, then $\lambda_{\alpha}(G')>\lambda_{\alpha}(G)$.
\end{lemma}

\begin{proof}
	By Equations
	%\ref{xax},
	(\ref{xax}) and $(\ref{rayleigh})$, we have
	\begin{align*}
	\lambda_{\alpha}(G')-\lambda_\alpha(G)&\geq X(G)^T(A_\alpha(G')-A_\alpha(G))X(G)\\&\geq 2(1-\alpha)k\sum_{w\in W}x_w(x_{v_1}-x_{u_1})+\alpha |W|k(x_{v_1}^2-x_{u_1}^2)\geq 0.
	\end{align*}
	By a similar argument as the proof of Lemma 2.2 in \cite{Xuejie-1}, we can deduce that $\lambda_{\alpha}(G')>\lambda_\alpha(G)$. These complete the proof.
\end{proof}

\begin{lemma}\label{bound}
	Let $G = K_{n_1}\vee K_{n_2}\vee K_{n_3}$. Then $\lambda_{\alpha}(G)$ is the largest root of $f(x)=0$, where
	\begin{align*}
	f(x)=&x^3+[3-(\alpha+1)(n_1+n_3)-(2\alpha+1)n_2]x^2+[n_2\alpha^2(n_1+n_2+n_3)+\alpha(n_1+2n_2+\\&n_3)(n_1+n_2+n_3-2)+3-2(n_1+n_2+n_3)+n_1n_3]x+[ \alpha^2(n_1n_2(1-n_1-2n_2)+\\&n_2n_3(1-n_3-2n_2)+n_2^2(1-n_2))+\alpha((1-n_3)n_1^2+(3n_2+2n_3-4n_2n_3-n_3^2-1)n_1+\\&(2n_2+n_3)(n_2+n_3-1))+n_1n_3(n_2+1)+1-(n_1+n-2+n_3)].
	\end{align*}
		In Particular, we have the following two results.
	\begin{itemize}
		\item [\rm(i)] If $n_1=n_3$, then $\lambda_{\alpha}(G)=\frac{1}{2}[(n_1+n_2-2)+\alpha(2n_1+n_2)+g(x)]$, where
		\begin{align*}
		g(x)=\sqrt{(2n_1+n_2)^2\alpha^2-2(2n_1^2+ 5n_1n_2+n_2^2)\alpha+n_1^2+6n_1n_2+n_2^2}.
		\end{align*}
		\item[\rm(ii)]	If $n_1=n_2=n_3=k\ (k\in Z^+)$, then $\lambda_{\alpha}(G)=\frac{3\alpha+2+\sqrt{9\alpha^2-16\alpha+8}}{2}k-1$.
	\end{itemize}
	\begin{proof}
		Note that all vertices in $V(K_{n_1})$, all vertices in $V(K_{n_2})$ and all vertices in $V(K_{n_3})$ are equivalent, respectively. Let $X(G)=(\overbrace{x_1,\ldots,x_1}^{n_1},\overbrace{x_2,\ldots,x_2}^{n_2},\overbrace{x_3,\ldots,x_3}^{n_3})^T$ be the Perron vector corresponding to $\lambda_{\alpha}(G)$, where $x_i=x_{v^i_1}=x_{v^i_2}=\ldots=x_{v^i_{n_i}}$ for each $i=1,2,3$. By Equation (\ref{AX}), we have
\begin{equation}\label{x-A}
\left\{
\begin{array}{ll}
[\lambda_{\alpha}(G)-\alpha(n_1+n_2-1)-(1-\alpha)(n_1-1)]x_1-(1-\alpha)n_2x_2=0,\\
-(1-\alpha)n_1x_1+[\lambda_{\alpha}(G)-\alpha(n_1+n_2+n_3-1)-(1-\alpha)(n_2-1)]x_2-\\(1-\alpha)n_3x_3=0,\\
-(1-\alpha)n_2x_2+[\lambda_{\alpha}(G)-\alpha(n_2+n_3-1)-(1-\alpha)(n_3-1)]x_3=0.
\end{array}
\right.
\end{equation}
		
		Let $f(x)$ be the determinant of the coefficient matrix of the Equation (\ref{x-A}). Notice that $X(G)\neq 0$, by using the theory of solving homogeneous linear equations, then $\lambda_\alpha(G)$ is the largest root of $f(x)=0$, where
        \begin{align*}
        f(x)=&x^3+[3-(\alpha+1)(n_1+n_3)-(2\alpha+1)n_2]x^2+\\&[(\alpha^2n_2-2)(n_1+n_2+n_3)+\alpha n_1(n_1+3n_2+2n_3-2)]x+\\&[\alpha n_2(2n_2+3n_3-4)+\alpha n_3(n_3-2)+n_1n_3+3]x+\\&(\alpha-\alpha n_3-\alpha^2n_2)n_1^2+ [-2\alpha^2n_2^2 + (3\alpha + n_3 - 4\alpha n_3 + \alpha^2)n_2]n_1+\\&[(n_3 - 1)(\alpha-\alpha n_3 + 1)]n_1-(\alpha n_2 + \alpha n_3 - 1)(\alpha n_2 - 1)(n_2 + n_3 - 1).
        \end{align*}
        \begin{itemize}
        	\item [\rm(i)]If $n_1=n_3$, then the vertices in $V(K_{n_1})\cup V(K_{n_3})$ are equivalent in $G$, that is $x_1=x_3$. similarly we have
			\begin{equation*}\label{x-B}
		\left\{
		\begin{array}{ll}
		[\lambda_{\alpha}(G)-\alpha(n_1+n_2-1)-(1-\alpha)(n_1-1)]x_1-(1-\alpha)n_2x_2=0,\\
		-2n_1(1-\alpha)x_1+[\lambda_{\alpha}(G)-\alpha(2n_1+n_2-1)-(1-\alpha)(n_2-1)]x_2=0.
		\end{array}
		\right.
		\end{equation*}
    	
    Similarly, %since $X(G)\neq 0$,
    according to the theory of solving homogeneous linear equations, by direct calculation, we have $\lambda_{\alpha}(G)=\frac{1}{2}[(n_1+n_2-2)+\alpha(2n_1+n_2)+g(x)]$, where
    \begin{align*}
    g(x)=\sqrt{(2n_1+n_2)^2\alpha^2-2(2n_1^2+ 5n_1n_2+n_2^2)\alpha+n_1^2+6n_1n_2+n_2^2}.
    \end{align*}

		\item [\rm(ii)] If $n_1=n_2=n_3=k\ (k\in Z^+)$, by direct calculation,  then
		\begin{align*}
	 \lambda_{\alpha}(G)=&\frac{1}{2}[(n_1+n_2-2)+\alpha(2n_1+n_2)+\\&\sqrt{(2n_1+n_2)^2\alpha^2-2(2n_1^2+ 5n_1n_2+n_2^2)\alpha+n_1^2+6n_1n_2+n_2^2}]\\=&\frac{3\alpha+2+\sqrt{9\alpha^2-16\alpha+8}}{2}k-1.
	 \end{align*}
		\end{itemize}
		  These complete the proof.
	\end{proof}	
\end{lemma}

\begin{lemma}\label{bound-2}
		Let $G = K_{n_1}\vee K_{n_2}\vee K_{n_2}\vee K_{n_1}$. Then $\lambda_{\alpha}(G)=\frac{1}{2}[(\alpha(n_1+n_2)+n_1+2n_2-2)+\alpha(2n_1+n_2)+g(x)]$, where
		$g(x)=\sqrt{(\alpha-1)^2n_1^2+2\alpha(\alpha-1)n_1n_2+(\alpha-2)^2n_2^2}$.
		In Particular, if $n_1=n_2=k\ (k\in Z^+)$, then $\lambda_{\alpha}(G)=\frac{(2\alpha+3+\sqrt{4\alpha^2-8\alpha+5})}{2}k-1$.
\end{lemma}
\begin{proof}
Note that all vertices in $V(K_{n_1})$ and all vertices in $V(K_{n_2})$ are equivalent, respectively. Let $X(G)$ be the Perron vector corresponding to $\lambda_{\alpha}(G)$. By symmetry, we have $X(G)=(\overbrace{x_1,\ldots,x_1}^{n_1},\overbrace{x_2,\ldots,x_2}^{n_2},\overbrace{x_2,\ldots,x_2}^{n_2},\overbrace{x_1,\ldots,x_1}^{n_1})^T$. By using Equation (\ref{AX}), we have
\begin{equation*}\label{x-C}
\left\{
\begin{array}{ll}
[\lambda_{\alpha}(G)-\alpha(n_1+n_2-1)-(1-\alpha)(n_1-1)]x_1-(1-\alpha)n_2x_2=0,\\
-(1-\alpha)n_1x_1+[\lambda_{\alpha}(G)-\alpha(n_1+2n_2-1)-(1-\alpha)(2n_2-1)]x_2=0.
\end{array}
\right.
\end{equation*}
Similar to the proof of Lemma \ref{bound}, by direct calculation, our results can be obtained easily.
\end{proof}
\section{Proof of Theorem \ref{th-main}}\label{main}
 Let $G$ be the graph with the maximum $A_\alpha$ spectral radius among all graphs in $\mathcal{G}_{n,k}^d$. We call the graph $G$ is \textit{maximal} in $\mathcal{G}_{n,k}^d$.

According to Lemmas \ref{d-critical} and \ref{le:g-uv},
the maximal graph $G$ must be diameter critical. Since adding edges increases the $A_\alpha$ spectral radius.  We have $G\cong K_{n_0}\vee\cdots \vee K_{n_d}$, where $n_0=n_1=1$ and $n_i\geq k$ for each $i=1,2,\ldots,d-1$, as shown in Figure \ref{critical}. The vertices sets of the subgraph $K_{n_i}$ in $G$ are simply denoted as $V_i$ for each $i=0,1,\ldots,d$. Therefore, $\{V_0,V_1,\ldots,V_d\}$ is a partition of $V(G)$.

\begin{figure}[!h]
	\centering
	\includegraphics[width=80mm]{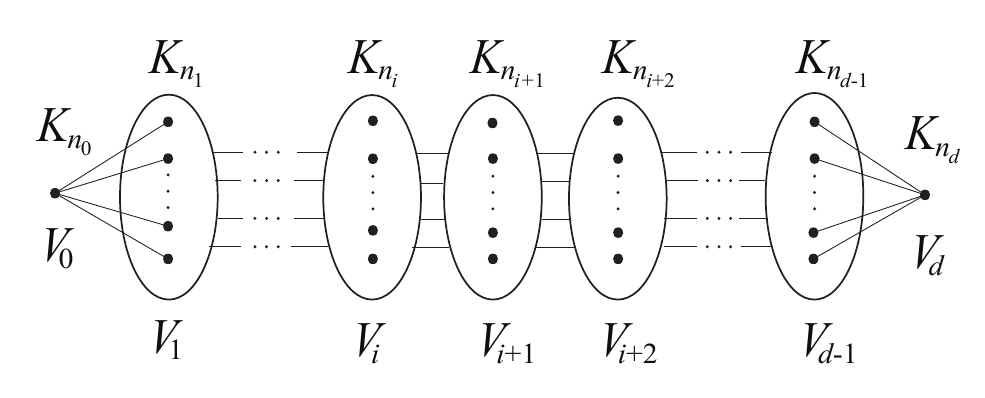}
	\caption{The graph $G=K_{n_0}\vee\cdots \vee K_{n_d}$}%, where $V_i=V(K_{n_i})$ for $0\leq i\leq d$.}
	\label{critical}
\end{figure}
 Note that any two distinct vertices in $V_1$ (resp. $V_2,\ldots,V_{d-1}$) are equivalent.
 Let $X(G)$ be the Perron vector corresponding to $\lambda_\alpha(G)$. Then we have $X(G)=(x_0,\overbrace{x_1,\ldots,x_1}^{n_1},\ldots,\\\overbrace{x_{d-1},\ldots,x_{d-1}}^{n_{d-1}},x_d)^T$.%, where $x_i=x_{v^i_1}=x_{v^i_2}=\ldots=x_{v^i_{n_i}}$ for each $1\leq i\leq d-1$.
 \begin{figure}[!h]
 	\centering
 	\includegraphics[width=80mm]{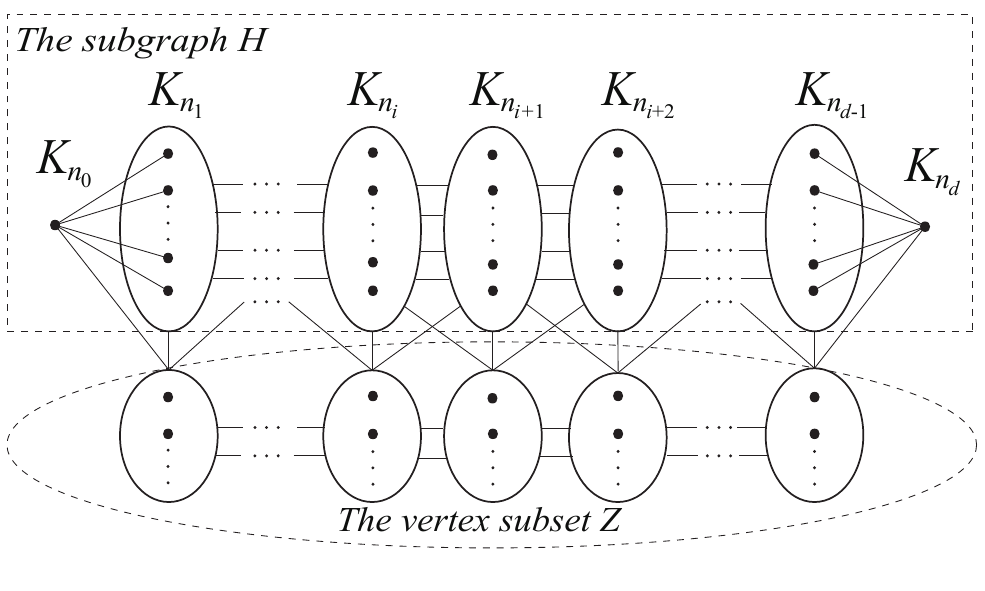}
 	\caption{The graph $G=K_{n_0}\vee\cdots \vee K_{n_d}$ with the subgraph $H$.}
 	\label{critical-H}
 \end{figure}

 Let $H\cong K_{1}\overbrace{\vee K_{k}\vee\cdots\vee K_{k}}^{d-1}\vee K_{1}$ be a subgraph of $G$ with vertex set $V(H)$. Let $Z=V(G)\setminus V(H)=\cup_{i=0}^{d}Z_i$, where $Z_i\subset V_i$, $i=0,1,\ldots,d$, as shown in Figure \ref{critical-H}.  Thus $\{V_0(H),V_1(H),\ldots,V_d(H)\}$ is a partition of $V(H)$, where $V_0(H)=V_0$, $V_d(H)=V_d$,
  $V_i(H)\subset V_i\subset V(G)$ and $|V_i(H)|=k$ for each $i=1,\ 2\ldots,d-1$. Similarly, $\{Z_0,Z_1,\ldots,Z_d\}$ is a partition of $Z$. It is possible that $Z_i=\emptyset$ for some $i\in[0,d]$, for instance $Z_0=Z_d=\emptyset$.

 Let $\mathcal{G}_1$ be the set of all graphs of order $n$, in which each graph is isomorphic to $K_{n_0}\vee\cdots \vee K_{n_d}$, where $n_0=n_d=1$, and there exists only one $j\in\{1,2,\ldots,d-1\}$ such that $n_j\geq k$ and $n_{i}=k$ for each $i\in\{2,\ldots,d-1\}\setminus\{j\}$. Obviously, $H\in\mathcal{G}_1\subseteq\mathcal{G}_{n,k}^d$.

Next, we keep the notations defined in the preceding section unless otherwise stated and prove several Lemmas.

\begin{lemma}\label{th-1}
Let $G$ be a maximal graph in $\mathcal{G}_{n,k}^d$. If $d\geq2$ and $k\geq 2$, then $G\in\mathcal{G}_1$.
\end{lemma}

\begin{proof}
For $d=2$ and $k\geq 2$, it is obvious that $G\in\mathcal{G}_1$. Next we consider $d\geq 3$ and $k\geq 2$.

By contradiction, we suppose $G\notin\mathcal{G}_1$, then there exists at least two $j_1$,\ldots, $j_s\in\{1,2,\ldots,d-1\}$ such that $Z_{j_1},\ldots,Z_{j_s}\neq\emptyset$ $(s\geq 2)$ and $Z_i=\emptyset$ for each $i\in\{1,2,\ldots,d-1\}\setminus\{j_1,\ldots, j_s\}$, which also implies that $|Z|\geq 2$.

  For $d=3$ and $k\geq 2$, if $G\notin\mathcal{G}_1$,  then $Z_1\neq\emptyset$ and $Z_2\neq\emptyset$. Let $G'$ be a graph obtained from $G$ by deleting all edges between $Z_{1}$ and $V_{0}$ and adding all edges between $Z_{1}$ and $V_{d}$. Without loss of generality, we assume $x_d\geq x_0$. Since $|V_0|=|V_d|=1$, by using Lemma \ref{le:g-uw+vw},
we have $\lambda_\alpha(G')>\lambda_\alpha(G)$. Note that $G'\in\mathcal{G}_1\subseteq\mathcal{G}_{n,k}^d$, which contradicts the assumption that $G$ is the maximal graph in $\mathcal{G}_{n,k}^k$. Thus $G\in\mathcal{G}_1$.

Next we consider $d\geq 4$ and $k\geq 2$.
Let $x_a\geq x_b\geq x_c$ be three largest values in the set $\{x_i|i=0,1,\ldots,d\}$ of the graph $G$, where $a$, $b$ and $c$ are pairwise distinct.  Next we consider the following two cases: %Case 1.$0$ or $d\notin\{a,b,c\}$; Case 2. $0$ or $d\in\{a,b,c\}$.
%By using the edge-shifting operation, in Case 1., we obtain graphs $G_{1,1}$, $G_{1,2}$, $G_2$, $G_{2,1}$, $G_{2,2}$, and $G_{2,3}$, such that $\lambda_\alpha(G_{2,3}),\lambda_{\alpha}(G_{2,2})>\lambda_{\alpha}(G_{2,1})>\lambda_{\alpha}(G_2),\lambda_{\alpha}(G_{1,2}),\lambda_{\alpha}(G_{1,1})>\lambda_{\alpha}(G)$, where $G_{2,3},G_{2,2}\in\mathcal{G}_1\subseteq\mathcal{G}_{n,k}^d$. This contradicts our assumption that $G$ is a maximal graph in $\mathcal{G}_{n,k}^d$ and $G\notin\mathcal{G}_1$. Similarly, in Case 2., we have $\lambda_{\alpha}(G_{3,4})>\lambda_{\alpha}(G_{3,3})>\lambda_{\alpha}(G_{3,2})>\lambda_{\alpha}(G_{3,1})$.
\begin{itemize}
\item[Case 1.]  $0$ or $d\notin\{a,b,c\}$.

	For any a vertex $v\in Z$, let $V_{a_v}(H),V_{b_v}(H),V_{c_v}(H)\in V(H)$ be three elements of the vertex partition set of $V(H)$, whose vertices are adjacent to $v$. It is possible that $\{a_v,b_v,c_v\}\cap\{a,b,c\}\neq\emptyset$. Because $\min\{x_a,x_b ,x_c\}\geq\max\{x_{a_v},x_{b_v},x_{c_v}\}$, without loss of generality, we assume $x_a\geq x_{a_v}$, $ x_b\geq x_{b_v}$ and $x_{c}\geq x_{c_v}$.
\begin{itemize}
		\item[Step 1.] Let $G_{1,1}$ be a graph obtained from $G$ by deleting all edges between $v$ and $V_{a_v}(H)$, $V_{b_v}(H)$, $V_{c_v}(H)$, respectively, and adding all edges between $v$ and $V_a(H)$, $V_b(H)$, $V_c(H)$, respectively.
		
		For $v\notin Z_1$ or $Z_{d-1}$, the graph $G$ and the graph $G_{1,1}$  ($v\notin Z_1$ or $Z_{d-1}$) is shown in Figure \ref{1.1} ($a$), ($b$), respectively. In Figure \ref{1.1} ($c$), we denote the edge-shifting operation from the graph $G$ to the graph $G_{1,1}$ ($v\notin Z_1$ or $Z_{d-1}$) as $G\rightarrow G_{1,1}$ ($v\notin Z_1$ or $Z_{d-1}$), where the thick dashed lines are the deleted edges in the graph $G$ and the thick solid lines are the added edges in the graph $G_{1,1}$.
		In Figures \ref{1.1}$-$\ref{g33-g34}, take Figure \ref{1.1} ($c$) as an example, we just show the edge-shifting operations between the corresponding two graphs.
%	\begin{figure}[!h]
	%	\centering
	%	\includegraphics[width=80mm]{Figures/1.2.pdf}
	%	\caption{ $G\rightarrow G_{1,1}$ ($v\in Z_1$ or $Z_{d-1}$).}
	%	\label{1.2}
%	\end{figure}
	
		If $v\notin Z_1$ or $Z_{d-1}$, by using Equations (\ref{xax}) and (\ref{rayleigh}), then
		\begin{align*}
		&\lambda_{\alpha}(G_{1,1})-\lambda_{\alpha}(G)\\\geq& X(G)^T[A_\alpha(G_{1,1})-A_\alpha(G)]X(G)\\
		=& 2k(1-\alpha)x_v(x_a+x_b+x_c-x_{a_v}-x_{b_v}-x_{c_v})+\\
		&\alpha k(x_a^2+x_b^2+x_c^2-x_{a_v}^2-x_{b_v}^2-x_{c_v}^2)\geq 0.
		\end{align*}
		If $v\in Z_1$ or $Z_{d-1}$, by symmetry, we can suppose $v\in Z_1$. By using Equations (\ref{xax}) and (\ref{rayleigh}), then we have
		\begin{align*}
		&\lambda_{\alpha}(G_{1,1})-\lambda_{\alpha}(G)\\\geq& X(G)^T[A_\alpha(G_{1,1})-A_\alpha(G)]X(G)\\
		=& 2k(1-\alpha)x_v(x_a+x_b+x_c-x_{0}-x_{1}-x_{2})+\alpha k(x_a^2+x_b^2+x_c^2-x_{0}^2-x_{1}^2-x_{2}^2)+\\&2(k-1)\big((1-\alpha)x_vx_{0}+\frac{\alpha}{2} x_{0}^2+\frac{\alpha}{2}x_v^2\big)\geq 0.
		\end{align*}
		Therefore, $\lambda_{\alpha}(G_{1,1})>\lambda_{\alpha}(G)$.
			 \begin{figure}[!h]
			\centering	
			\subfigure[The graph $G$ ($v\notin Z_1$ or $Z_{d-1}$).]{\includegraphics[width=70mm]{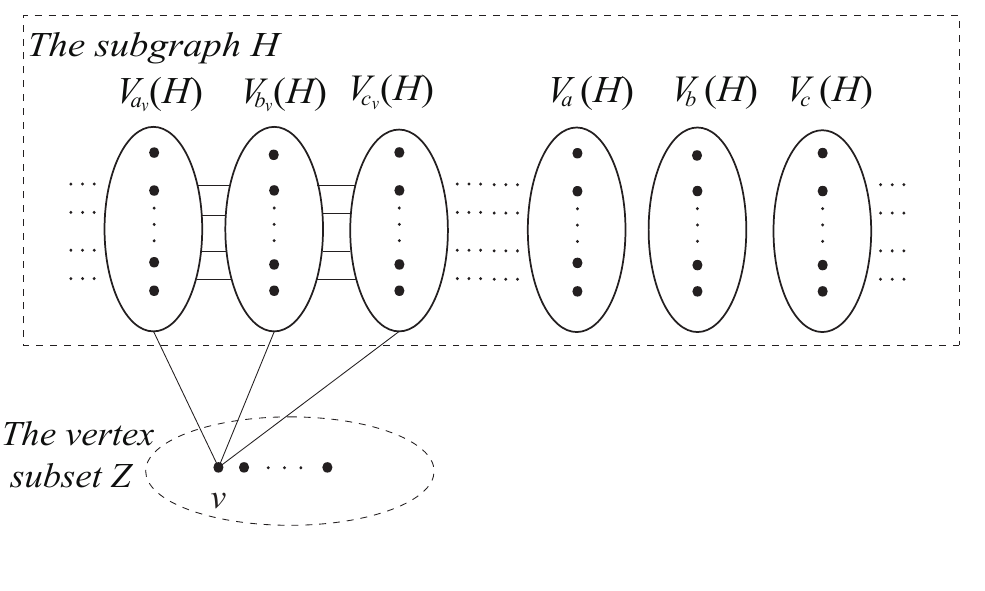}}
			\hspace{5mm}
			\subfigure[The graph $G_{1,1}$ ($v\notin Z_1$ or $Z_{d-1}$).]{\includegraphics[width=70mm]{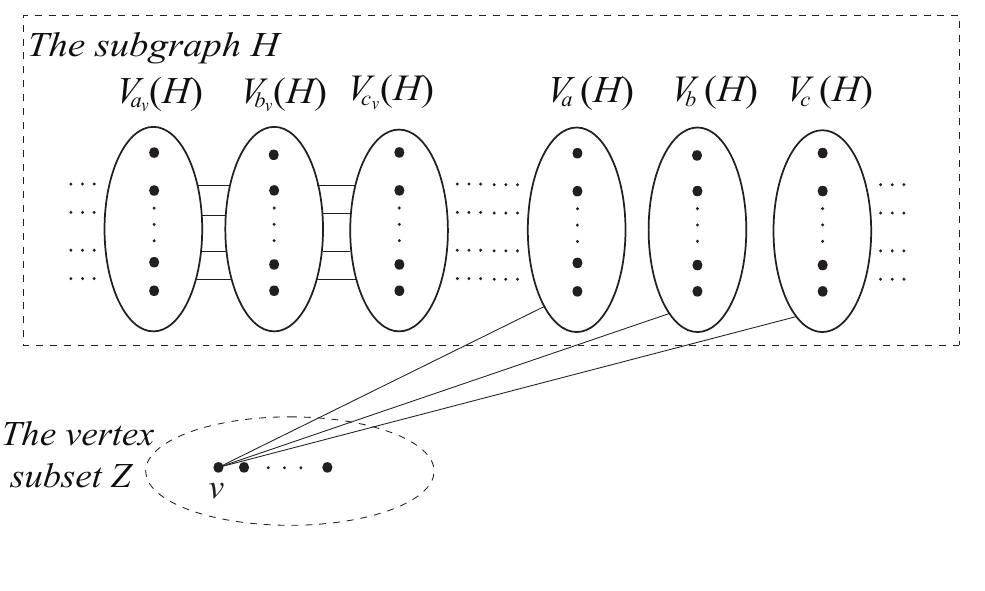}}
			\subfigure[The edge-shifting operations from $G$ to $G_{1,1}$ ($v\notin Z_1$ or $Z_{d-1}$).]{\includegraphics[width=70mm]{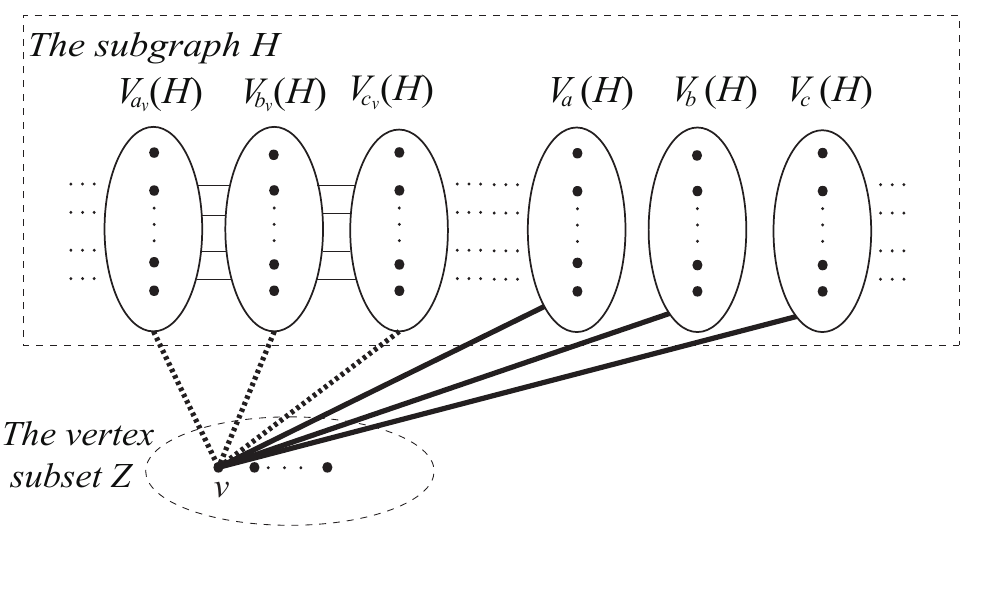}}
			\caption{$G\rightarrow G_{1,1}$ ($v\notin Z_1$ or $Z_{d-1}$).}
			\label{1.1}
		\end{figure}
\item[Step 2.] Let $w$ be a vertex of $G$ and $w\in Z\setminus \{v\}$.\\If $w\notin Z_1$ or $Z_{d-1}$, let $V_{a_w}(H),V_{b_w}(H),V_{c_w}(H)\in V(H)$  be three elements of the vertex partition set of $V(H)$, whose vertices are adjacent to $w$.  According to our assumption, $\min\{x_a,x_b,x_c\}\geq\max\{x_{a_w},x_{b_w},x_{c_w}\}$, thus without loss of generality, we assume
$x_a\geq x_{a_w}$, $ x_b\geq x_{b_w}$, $x_{c}\geq x_{c_w}$. Next we do the same operation as Step 1. and obtain a graph $G_{1,2}$ from $G_{1,1}$ by deleting all edges between $w$ and $V_{a_w}(H)$, $V_{b_w}(H)$, $V_{c_w}(H)$, respectively, and adding all edges between $w$ and $V_a(H)$, $V_b(H)$, $V_c(H)$, respectively. Then
		\begin{align*}	
			&\lambda_{\alpha}(G_{1,2})-\lambda_{\alpha}(G)\\\geq& X(G_{1,2})^TA_\alpha(G_{1,2})X(G_{1,2})-X(G)^TA_\alpha(G)X(G)\\
			\geq& X(G)^T[A_\alpha(G_{1,2})-A_\alpha(G)]X(G)\\
	=&X(G)^T[A_\alpha(G_{1,2})-A_\alpha(G_{1,1})]X(G)+X(G)^T[A_\alpha(G_{1,1})-A_\alpha(G)]X(G)\\
    \geq& 2k(1-\alpha)x_{w}(x_a+x_b+x_c-x_{a_w}-x_{b_w}-x_{c_w})+\\&\alpha k(x_a^2+x_b^2+x_c^2-x_{a_w}^2-x_{b_w}^2-x_{c_w}^2)\geq 0.
		\end{align*}
		If $w\in Z_1$ or $Z_{d-1}$, by symmetry, we suppose $w\in Z_1$. By using Equations (\ref{xax}) and (\ref{rayleigh}), then we have
			\begin{align*}	
		&\lambda_{\alpha}(G_{1,2})-\lambda_{\alpha}(G)\\
		=&\lambda_{\alpha}(G_{1,2})-\lambda_{\alpha}(G_{1,1})+\lambda_{\alpha}(G_{1,1})-\lambda_{\alpha}(G)\\
		\geq&X(G_{1,2})^TA_\alpha(G_{1,2})X(G_{1,2})-X(G)^TA_\alpha(G)X(G)\\
		\geq& X(G)^T[A_\alpha(G_{1,2})-A_\alpha(G)]X(G)\\
		=&X(G)^T[A_\alpha(G_{1,2})-A_\alpha(G_{1,1})]X(G)+X(G)^T[A_\alpha(G_{1,1})-A_\alpha(G)]X(G)\\ \geq&2k(1-\alpha)x_{w}(x_a+x_b+x_c-x_{a_w}-x_{b_w}-x_{c_w})+\\&\alpha k(x_a^2+x_b^2+x_c^2-x_{a_w}^2-x_{b_w}^2-x_{c_w}^2)
		+\\&2(k-1)\big((1-\alpha)x_w x_{0}+\frac{\alpha}{2} x_{0}^2+\frac{\alpha}{2}x_w^2\big)\geq 0.
		\end{align*}
		Repeat this step until all of the vertices in $Z$ have been chosen. Thus in the final graph, denoted by $G_2$, all vertices in $Z$ are adjacent to $V_a(H)$, $V_b(H)$ and $V_c(H)$, and $\lambda_{\alpha}(G_2)>\lambda_{\alpha}(G)$.
		
In order to ensure that the maximal graph $G$ after the edge-shifting operation is still a diameter critical graph, and prove that our conclusions hold, we do the Step 3.

\item[Step 3.] Let $G_{2,1}$ be the graph obtained
from $G_2$ by adding some edges so that $G_{2,1}[Z]$ is a clique. Then $\lambda_\alpha(G_{2,1})>\lambda_{\alpha}(G_2)>\lambda_{\alpha}(G)$.\\ If $G[V_a(H)\cup V_b(H)\cup V_c(H)]=K_k\vee K_k\vee K_k$, then our results hold. Next we consider the following two cases.% (1)  $G[V_a(H)\cup V_b(H)\cup V_c(H)]=(K_k\vee K_k)\cup K_k$. (2) $G[V_a(H)\cup V_b(H)\cup V_c(H)]=K_k\cup K_k\cup K_k$.
	\begin{itemize}
     \item[Case 1.1.] 	
     \begin{figure}[!h]
     	\centering
     	\includegraphics[width=80mm]{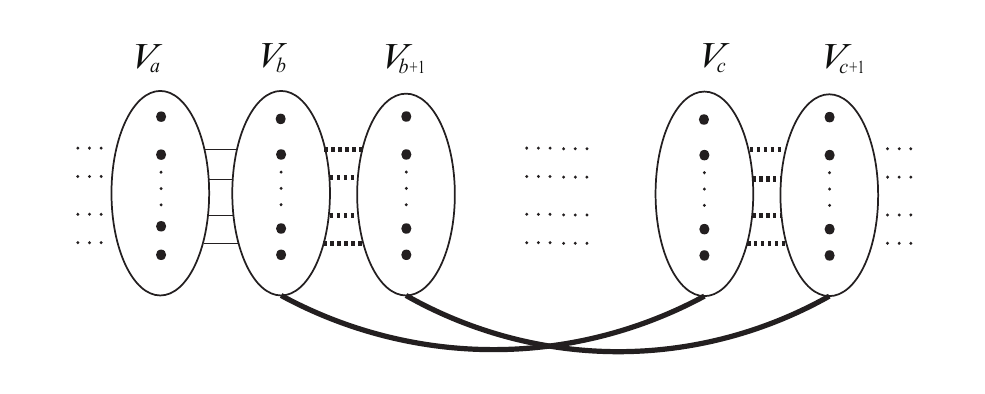}
     	\caption{ $G_{2,1}\rightarrow G_{2,2}$.}
     	\label{3.21}
     \end{figure}
$G[V_a(H)\cup V_b(H)\cup V_c(H)]=(K_k\vee K_k)\cup K_k$.

Let $G_{2,2}$ be the graph
     obtained from $G_{2,1}$ by deleting all edges between $V_b(H)$ and $V_{b+1}(H)$, all edges between $V_{c}(H)$ and $V_{c+1}(H)$, and adding all edges between $V_b(H)$ and $V_c(H)$, all edges between $V_{b+1}(H)$ and $V_{c+1}(H)$, as shown in Figure \ref{3.21}.

     If $c+1\neq d$, then
    \begin{align*}
    \lambda_{\alpha}(G_{2,2})-\lambda_{\alpha}(G_{2,1})\geq 2(1-\alpha)k^2(x_b-x_{c+1})(x_c-x_{b+1})\geq 0.
    \end{align*}
    If $c+1=d$, then
     \begin{align*}
   \lambda_{\alpha}(G_{2,2})-\lambda_{\alpha}(G_{2,1})\geq& 2(1-\alpha)k^2(x_b-x_{c+1})(x_c-x_{b+1})+\\&2(1-\alpha)k(k-1)(x_{c+1}x_c-x_{b+1}x_{c+1})+\\&\alpha k(k-1) (x_{c}^2-x_{b+1}^2)\geq 0.
    \end{align*}
   Then $\lambda_{\alpha}(G_{2,2})>\lambda_{\alpha}(G_{2,1})>\lambda_{\alpha}(G_2)>\lambda_{\alpha}(G)$, and $G_{2,2}\in\mathcal{G}_1\subset\mathcal{G}_{n,k}^d$. This contradicts that $G\notin\mathcal{G}_1$ is the maximal graph in $\mathcal{G}_{n,k}^k$. Thus the maximal graph $G\in\mathcal{G}_1\subset\mathcal{G}_{n,k}^d$.
   \item[Case 1.2.]
   \begin{figure}[!h]
   	\centering
   	\includegraphics[width=80mm]{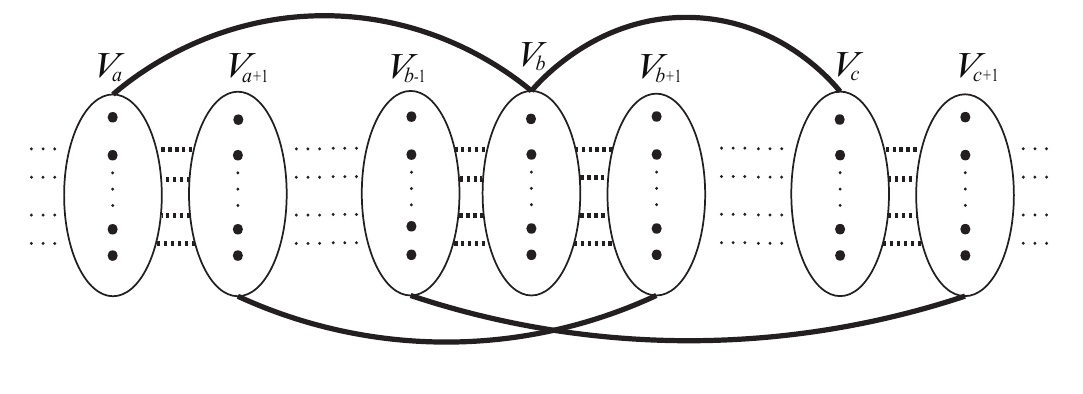}
   	\caption{ $G_{2,1}\rightarrow G_{2,3}$.}
   	\label{3.2-2}
   \end{figure}
   $G[V_a(H)\cup V_b(H)\cup V_c(H)]=K_k\cup K_k\cup K_k$. \\Let $G_{2,3}$ be the graph obtained
from $G_{2,1}$ by deleting all edges between $V_a(H)$ and $V_{a+1}(H)$, all edges between $V_b(H)$ and $V_{b+1}(H)$, all edges between $V_{c}(H)$ and $V_{c+1}(H)$, all edges between $V_b(H)$ and $V_{b-1}(H)$, and adding all edges between $V_a(H)$ and $V_b(H)$, all edges between $V_{b}(H)$ and $V_{c}(H)$, all edges between $V_{a+1}(H)$ and $V_{b+1}(H)$, all edges between $V_{b-1}(H)$ and $V_{c+1}(H)$, as shown in Figure \ref{3.2-2}.

If $c+1\neq d$, then
\begin{align*}
	 &\lambda_{\alpha}(G_{2,3})-\lambda_{\alpha}(G_{2,1})\\
	\geq&2(1-\alpha)k^2[(x_a-x_{b+1})(x_b-x_{a+1})+(x_b-x_{c+1})(x_c-x_{b-1})]\geq 0.
	\end{align*}
If $c+1=d$, then
 \begin{align*}
 &\lambda_{\alpha}(G_{2,3})-\lambda_{\alpha}(G_{2,1})\\
 \geq&2(1-\alpha)k^2[(x_a-x_{b+1})(x_b-x_{a+1})+(x_b-x_{c+1})(x_c-x_{b-1})]+\\&2(1-\alpha)k(k-1)(x_cx_{c+1}-x_{b-1}x_{c+1})+\alpha k(k-1)(x_c^2-x_{b-1}^2)\geq 0.
 \end{align*}
Then $\lambda_{\alpha}(G_{2,3})>\lambda_{\alpha}(G_{2,1})>\lambda_{\alpha}(G)$ and $G_{2,3}\in\mathcal{G}_1\subset\mathcal{G}_{n,k}^d$. This contradicts that $G\notin\mathcal{G}_1$ is the maximal graph in $\mathcal{G}_{n,k}^d$. Thus the maximal graph $G\in\mathcal{G}_1\subset\mathcal{G}_{n,k}^d$.
\end{itemize}
\end{itemize}
\item[Case 2.] $0$ or $d\in\{a,b,c\}$.
\\Since $x_{0}=\frac{(1-\alpha)|V_1|}{\lambda_\alpha-\alpha |V_1|}x_{1}$, by using Proposition 36. in \cite{Nik-1} and Lemma \ref{le:g-uv}, then $\lambda_{\alpha}>(|V_1|+1)-1=|V_1|$, that is $x_0<x_1$. Because of symmetry, we have $x_d<x_{d-1}$. Thus, $\{a,b,c\}=\{0,1,r\}$ or $\{r,d-1,d\}$, where $1\leq r\leq d-1$. Without loss of generality, we assume $\{a,b,c\}=\{0,1,r\}$,  where $2\leq r\leq d-1$. Next we consider two cases: first, $Z_1=Z_{d-1}=\emptyset$, second, $Z_1\neq Z_{d-1}\neq\emptyset$.  The case of $Z_1\neq Z_{d-1}\neq\emptyset$ can be converted to the case that $Z_1\neq \emptyset$ and $Z_d\neq \emptyset$   by using some edge-shifting operations.
\begin{itemize}
\begin{figure}[!h]
	\centering
	\includegraphics[width=80mm]{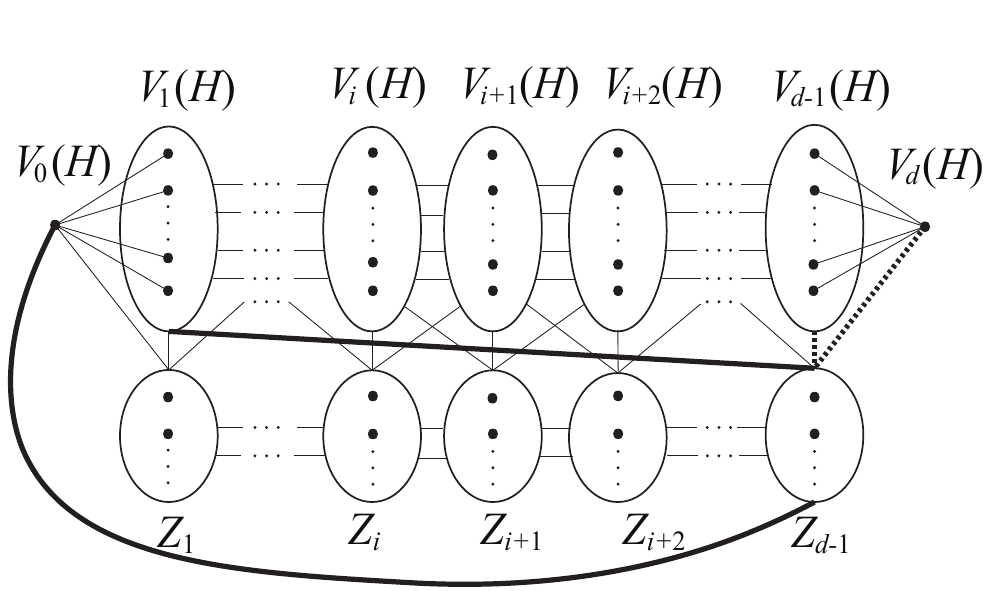}
	\caption{$G\rightarrow G_{3,1}$.}
	\label{g-31}
\end{figure}
\begin{figure}[!h]
	\centering
	\includegraphics[width=80mm]{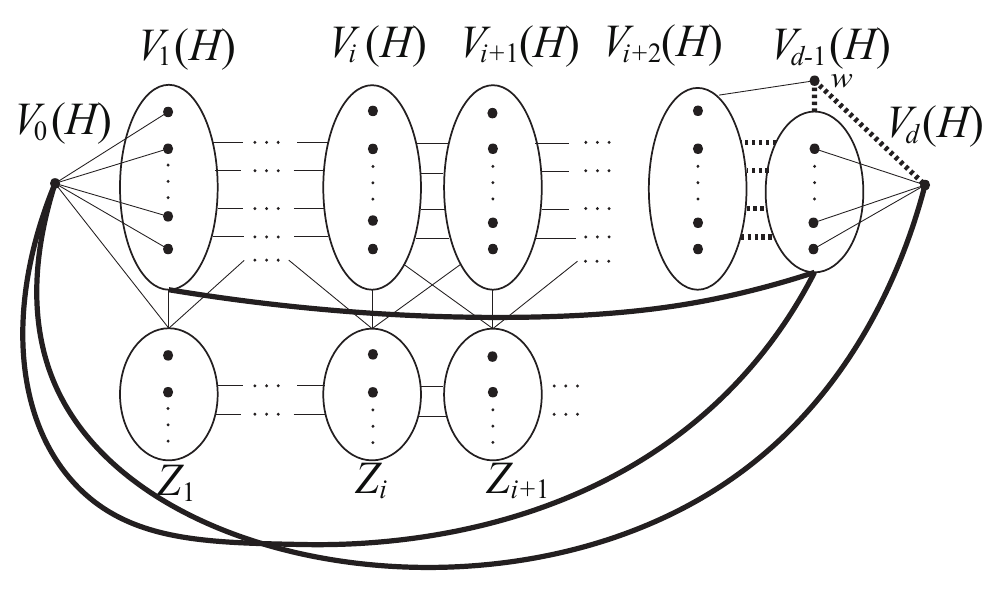}
	\caption{$G_{3,1}\rightarrow G_{3,2}$.}
	\label{g31-g32}
\end{figure}
\begin{figure}[!h]
	\centering
	\includegraphics[width=80mm]{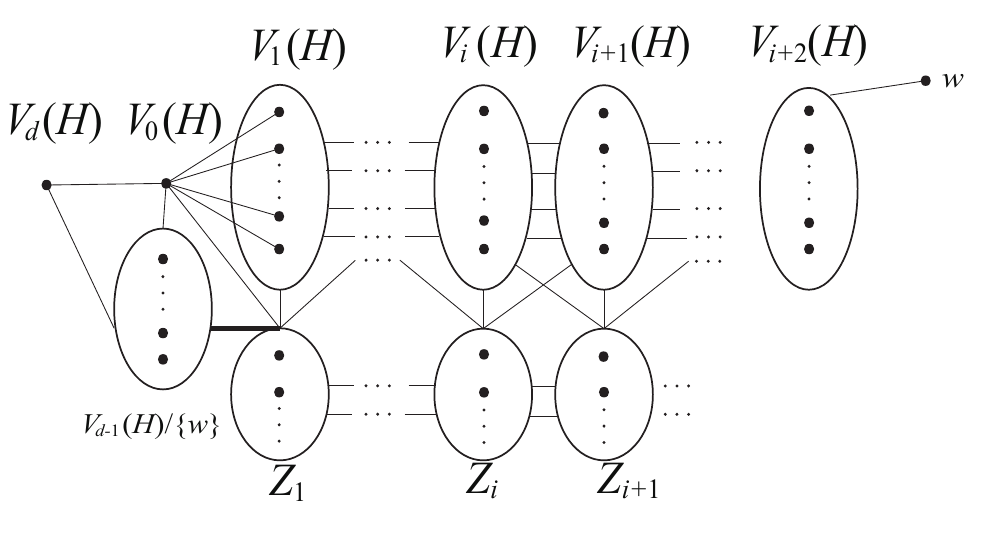}
	\caption{$G_{3,2}\rightarrow G_{3,3}$.}
	\label{g32-g33}
\end{figure}
\begin{figure}[!h]
	\centering
	\includegraphics[width=80mm]{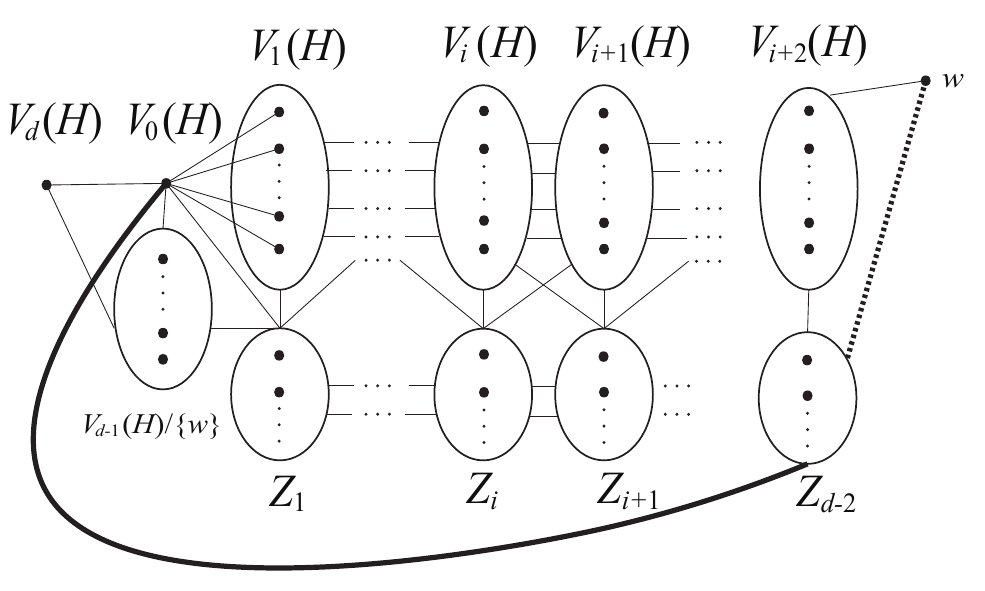}
	\caption{$G_{3,3}\rightarrow G_{3,4}$.}
	\label{g33-g34}
\end{figure}
\begin{figure}[!h]
	\centering
	\includegraphics[width=80mm]{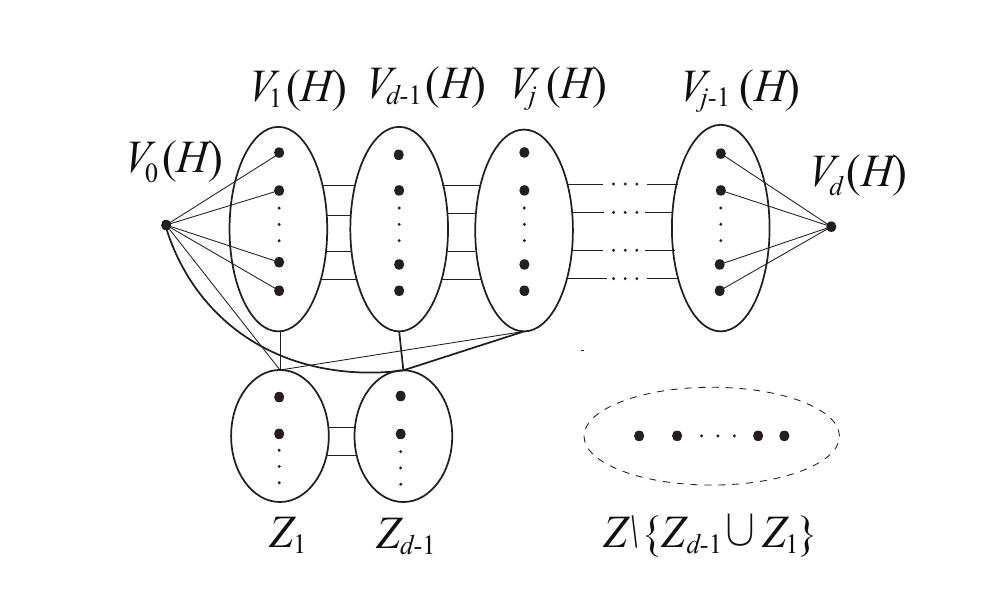}
	\caption{$G_{3,5}$.}
	\label{g35}
\end{figure}
\begin{figure}[!h]
	\centering
	\includegraphics[width=80mm]{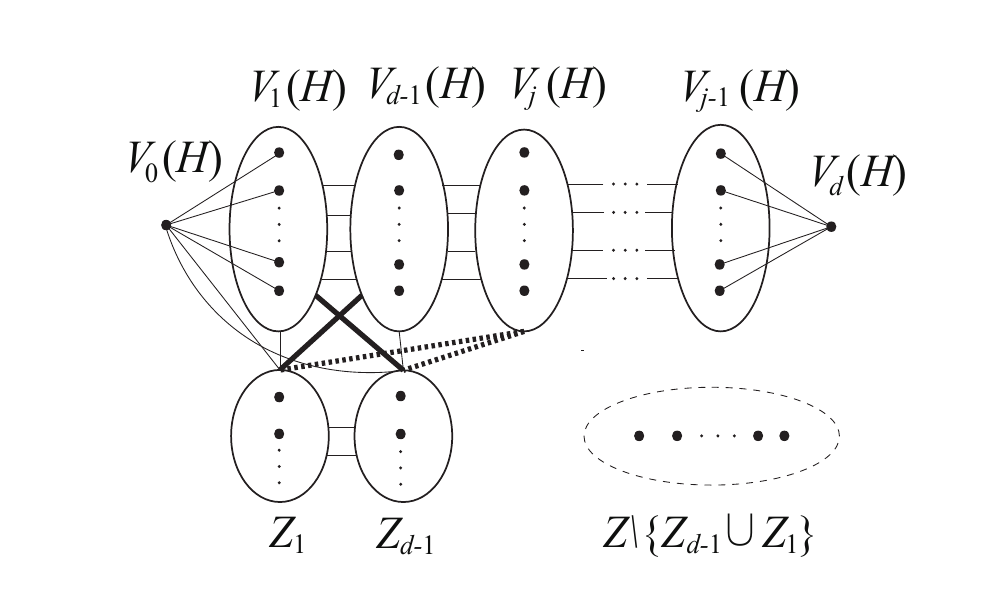}
	\caption{$G_{3,5}\rightarrow G_{3,6}$.}
	\label{g35-g36}
\end{figure}
	\item[Case 2.1.]  $Z_1=Z_{d-1}=\emptyset$. \\Let $x_{a'}\geq x_{b'}\geq x_{c'}$ be the three largest components of  $X(G)\setminus\{x_0,x_d\}=(x_1,x_2,\ldots,x_{d-1})^T$ corresponding to all vertices of $V_{a'}$, $V_{b'}$, $V_{c'}$ ($a'$, $b'$ and $c'$ are pairwise distinct), respectively. We follow all steps as Case 1., then we can obtain the graphs like $G_{1,1}$, $G_{1,2}$, $G_{2,1}$, $G_2$, $G_{2,2}$ or $G_{2,3}$, such that $\lambda_\alpha(G_{2,3}) \big(\lambda_{\alpha}(G_{2,2})\big)>\lambda_{\alpha}(G_{2,1})>\lambda_{\alpha}(G_2) \big(\lambda_{\alpha}(G_{1,2}),\lambda_{\alpha}(G_{1,1})\big)>\lambda_{\alpha}(G)$, where $G_{2,3}$ and $G_{2,2}$ are in $\mathcal{G}_1\subseteq\mathcal{G}_{n,k}^d$. Then we obtain a same contradiction like Case 1. that $G$ is a maximal graph in $\mathcal{G}_{n,k}^d$ and $G\notin\mathcal{G}_1$.
	
	\item[Case 2.2.]  $Z_1\neq Z_{d-1}\neq\emptyset$ and $r\neq d-1$. \\Let $G_{3,1}$ be a graph obtained from $G$ by deleting all edges between $Z_{d-1}$ and $V_d(H)$, $V_{d-1}(H)$, respectively, and adding all edges between $Z_{d-1}$ and $V_0(H)$, $V_1(H)$, respectively, as shown in Figure \ref{g-31}. Then there is no edge between $Z$ and $V_{d-1}(H)\cup V_d(H)$.  By using Lemma \ref{v-u}, then $\lambda_{\alpha}(G_{3,1})>\lambda_{\alpha}(G)$. The following discussion also applies to the case that one of $Z_1$, $Z_{d-1}$ is not an empty set.
	
	First, we choose one vertex $w$ from $V_{d-1}(H)$. Let $G_{3,2}$ be a graph obtained from $G$ by deleting all edges between $w$ and $(V_{d-1}(H)\setminus{w})\cup V_d(H)$, all edges between $V_{d-2}(H)$ and $V_{d-1}(H)\setminus{w}$, and adding all edges between $V_{d-1}(H)\setminus{w}$ and $V_1(H)\cup V_0(H)$, and the edge $V_0V_d$, see Figure \ref{g31-g32}. By using Lemma \ref{v-u}, we have $\lambda_{\alpha}(G_{3,2})>\lambda_\alpha(G)$.
	
	 Second, let $G_{3,3}$ be a graph obtained from $G_{3,2}$ by adding all edges between $V_{d-1}\setminus{w}$ and $Z_1$, as shown in Figure \ref{g32-g33}. By using Lemma \ref{le:g-uv}, we have $\lambda_{\alpha}(G_{3,3})>\lambda_{\alpha}(G)$.
	
	  Third, if $Z_{d-2}=\emptyset$, then we may follow the proof of Case 1., otherwise, if $Z_{d-2}\neq\emptyset$, let $G_{3,4}$ be a graph obtained from $G_{3,3}$ by deleting all edges between $w$ and $Z_{d-2}$, and adding all edges between $V_0$ and $Z_{d-2}$, as shown in Figure \ref{g33-g34}. By using Lemma \ref{v-u}, we have $\lambda_{\alpha}(G_{3,4})>\lambda_{\alpha}(G)$. Next, we may follow the proof of Case 1. and get a contraction. Thus the maximal graph $G\in\mathcal{G}_1\subset\mathcal{G}_{n,k}^d$.	\item[Case 2.3.]  $Z_1\neq Z_{d-1}\neq\emptyset$ and $r=d-1$ ($\{a,b,c\}=\{0,1,r\}$).\\
	  First, we keep all edges between $Z_1$ and $V_0\cup V_1(H)$, keep all edges between $Z_{d-1}$ and $V_{d-1}(H)$. Next we delete all edges between $Z_{d-1}$ and $V_{d}$ and add all edges between $Z_{d-1}$ and $V_{0}$.
	
	   Second, let $x_j$ be the largest component of $X(G)\setminus\{x_0,x_1,x_{d-1},x_d\}$. Then we repeat the similar operations of Step 1.--Step 3. in Case 1. and obtain a graph $G_{3,5}$ with $\lambda_{\alpha}(G_{3,5})>\lambda_{\alpha}(G)$, as shown in Figure \ref{g35}. Let $G_{3,6}$ be a graph obtained from $G_{3,5}$ by deleting all edges between $Z_1$ and $V_j(H)$, all edges between $Z_{d-1}$ and $V_j(H)$, and adding all edges between $Z_1$ and $V_{d-1}(H)$, all edges between $Z_{d-1}$ and $V_1(H)$, as shown in Figure \ref{g35-g36}. Then by using Lemma \ref{le:g-uv}, we have $\lambda_{\alpha}(G_{3,5})>\lambda_{\alpha}(G)$.
	   Note that the graph $G_{3,6}$ satisfys $G_{3,6}\in\mathcal{G}_{n.k}^d$, and in the graph $G_{3,6}$, one of $Z_1$ and $Z_{d-1}\neq\emptyset$ and $r\neq d-1$. It belongs to Case 2.2.
\end{itemize}
\end{itemize}
\begin{comment}
Furthermore,
it is clearly that $x_{j-1}$, $x_{j}$, and $x_{j+1}$ are the three largest values in the set $\{x_i|\ 0\leq i\leq d\}$. Otherwise, we can obtain a maximal graph in $\mathcal{G}_{n,k}^d$ by following the three steps described above.
\end{comment}
These complete the proof of Lemma \ref{th-1}.
\end{proof}
\begin{figure}[!h]
	\centering
	\includegraphics[width=80mm]{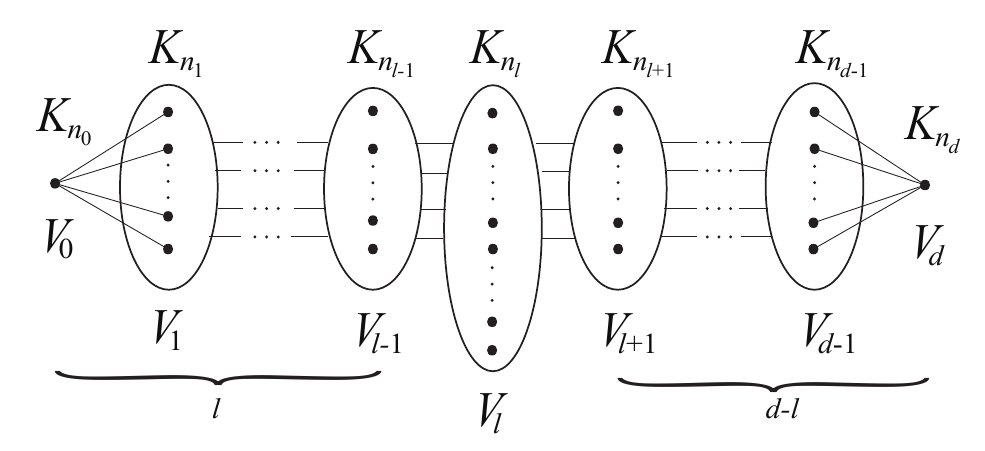}
	\caption{$G(l,d-l)$}%=\overbrace{K_{n_0}\vee\cdots\vee K_{n_{l-1}}}^l\vee K_{n_l}\vee\overbrace{K_{n_{l+1}}\cdots\vee K_{n_d}}^{d-l}$.}
	\label{fig1}
\end{figure}
\begin{lemma}\label{th-2}
Let $G=G(l,d-l)=\overbrace{K_{n_0}\vee\cdots\vee K_{n_{l-1}}}^l\vee K_{n_l}\vee\overbrace{K_{n_{l+1}}\cdots\vee K_{n_d}}^{d-l}$, where $n_0=n_d=1$, $n_{l}\geq 2k$, $n_i=k$  for each $i\in\{1,2,\ldots,d-1\}\setminus\{l\}$, as shown in Figure \ref{fig1}.
 Let $X(G)=(x_0,\overbrace{x_1,\ldots,x_1}^{n_1},\ldots,\overbrace{x_{d-1},\ldots,x_{d-1}}^{n_{d-1}},x_d)^T$ be the Perron vector of the graph $G$ corresponding to $\lambda_{\alpha}(G)$, where $x_i=x_{v^i_1}=x_{v^i_2}=\ldots=x_{v^i_{n_i}}$ for each $i=1,2,\ldots,d-1$. If $\alpha\in[0,1)$, $n_l\geq 2k$, and  $l\geq d-l+1$. Then
\begin{itemize}
	\item[\rm(i)] $x_{l-1}>\cdots>x_1>x_0$ and  $x_{l+1}>\cdots>x_{d-1}>x_d$.
	\item[\rm(ii)] $n_lx_l>kx_{l-1}$.
	\item[\rm(iii)] $x_d>x_0$ and $x_i<x_{d-i}$ for each $i=1,2,\ldots,d-l-1$, where $l\geq d-l+2$.
\end{itemize}
\end{lemma}
\begin{proof}
We keep all notations above unless otherwise stated. Let $\lambda_{\alpha}(G)=\lambda_{\alpha}$. Since $n_l\geq 2k$, by Lemmas \ref{Nik-complete} and \ref{le:g-uv}, then $\lambda_{\alpha}>\lambda_{\alpha}(K_k\vee K_{2k})=3k-1$. By Equation (\ref{AX}), we have $\lambda_{\alpha}x_0=\alpha kx_0+(1-\alpha)kx_1$, then $x_{0}=\frac{(1-\alpha)k}{\lambda_\alpha-\alpha k}x_{1}<x_{1}$ for any $\alpha\in[0,1)$.

 \textbf{(i) \bm{$x_{l-1}>\cdots>x_1>x_0$} and  \bm{$x_{l+1}>\cdots>x_{d-1}>x_d$}.}

 In terms of symmetry, we just consider $i=0,1,\ldots,l-1$. It is similar for $i=l+1,l+2,\ldots,d$. Next we are going to consider the following two cases:

\textbf{Case 1.} $l\geq 4$.

 By the eigenequations of $A_\alpha(G)$, we have the following recurrence equations:
\begin{equation}\label{recurrence-1}
(1-\alpha)kx_{i+1}+(2k\alpha+k-1-\lambda_\alpha) x_{i}+(1-\alpha)kx_{i-1}=0, \quad i=1,2,\ldots,l-1.
\end{equation}
Let $x_1,x_2,\ldots,x_{l-1}$ be a sequence of numbers
satisfying the above recursive formula. Thus its characteristic equation is
\begin{equation}\label{recurrence-2}
(1-\alpha)kt^2+(2k\alpha+k-1-\lambda_\alpha) t+(1-\alpha)k=0.
\end{equation}
Notice that $\lambda_\alpha>3k-1$, then $(2k\alpha+k-1-\lambda_\alpha)^{2}-4(1-\alpha)^{2}k^2>0$, thus Equation \eqref{recurrence-2} has two different real roots $t_{1}$ and $t_{2}$ such that
$$
t_{1} t_{2}=1, \quad t_{1}+t_{2}=\frac{\lambda_\alpha-2k\alpha+1-k}{(1-\alpha)k}>\frac{(3k-1)-2k\alpha+1-k}{(1-\alpha)k}=2.
$$
It is clear that $t_{1}>t_{2}>0$. Since $\lambda_{\alpha}>3k-1$, if $t=1$,  then the Equation \eqref{recurrence-2} is equal to $3k-1-\lambda_{\alpha}<0$. By Lemmas \ref{le:g-uv}, \ref{bound}(i), for $l\geq d-l+1\geq 5$, we have $\lambda_{\alpha}>\lambda_{\alpha}(K_k\vee K_{2k}\vee K_k)=\frac{4\alpha +3+\sqrt{16\alpha^2-32\alpha+17}}{2}k-1$, if $t=2$, then the Equation \ref{recurrence-2} is equal to $(7-\alpha)k-2-2\lambda_{\alpha}<(4 - 5\alpha - \sqrt{16\alpha^2 - 32\alpha + 17})k<0$, that is $t_1>2$. Therefore, we have $t_{1}>2,\ 1>t_{2}>0$.  By the theory of linear recurrence equation about the sequence, there exist $A$ and $B$ such that $x_i=At_1^{i-1}+Bt_2^{i-1}$, $i=2,3,\ldots,l-2$. The boundary conditions are as follows:
\begin{equation*}
\left\{
\begin{array}{c}
\begin{aligned}
 x_1&=A+B,\\
 x_2&=At_1+Bt_2,\\
\lambda_{\alpha}x_1&=2\alpha kx_1+(1-\alpha)(x_0+(k-1)x_1+kx_2).
\end{aligned}
\end{array}
\right.
\end{equation*}
Then
\begin{equation*}
\left\{
\begin{array}{c}
\begin{aligned}
A&= \frac{[\lambda_\alpha-(1+\alpha)k+(1-\alpha)(1-t_2k)]x_1-(1-\alpha)x_0}{(1-\alpha)(t_1-t_2)k},\\
B&=\frac{-[\lambda_\alpha-(1+\alpha)k+(1-\alpha)(1-t_1k)]x_1+(1-\alpha)x_0}{(1-\alpha)(t_1-t_2)k}.
\end{aligned}
\end{array}
\right.
\end{equation*}
That is,
$$x_i=\frac{(t_1^{i-1}-t_2^{i-1})\big[(\lambda_{\alpha}-(1+\alpha)k)x_1+(1-\alpha)(x_1-x_0)\big] -k(1-\alpha)(t_1^{i-2}-t_2^{i-2})x_1}{(1-\alpha)(t_1-t_2)k}.$$
Thus $$\frac{x_{i}}{x_{i-1}}=\frac{(t_1^{i-1}-t_2^{i-1})\big[(\lambda_{\alpha}-(1+\alpha)k)x_1+(1-\alpha)(x_1-x_0)\big] -k(1-\alpha)(t_1^{i-2}-t_2^{i-2})x_1}{(t_1^{i-2}-t_2^{i-2})\big[(\lambda_{\alpha}-(1+\alpha)k)x_1+(1-\alpha)(x_1-x_0)\big]-k(1-\alpha)(t_1^{i-3}-t_2^{i-3})x_1},$$
where $i=2,\ldots,l-1.$

Since $t_1t_2=1$, $t_1>2$ and $1>t_2>0$ for any $\alpha\in[0,1)$, then $t_1^{i-1}-t_2^{i-1}=t_1(t_1^{i-2}-t_2^{i})>t_1(t_1^{i-2}-t_2^{i-2})>2(t_1^{i-2}-t_2^{i-2})>0$.
% and $t_1^{i-1}-t_2^{i-1}=t_2(t_1^{i}-t_2^{i-2})>t_2(t_1^{i-2}-t_2^{i-2})>0$.

By using Lemma \ref{bound}(i), we have $\lambda_{\alpha}>\frac{4\alpha +3+\sqrt{16\alpha^2-32\alpha+17}}{2}k-1$. Then for $k\geq 2$, then we have
$[\lambda_{\alpha}-(1+\alpha)k]-k>~\frac{2\alpha -1+\sqrt{16\alpha^2-32\alpha+17}}{2}k-1>~0.93k-1>~0.$
Thus $\lambda_{\alpha}-(1+\alpha)k>k>(1-\alpha)k$, for any $\alpha\in[0,1)$.
\begin{comment}
\begin{align*}
&~[(t_1^{i-1}-t_2^{i-1})(\lambda_{\alpha}-(1+\alpha)k)-(1-\alpha)(t_1^{i-2}-t_2^{i-2})k]- [(t_1^{i-2}-t_2^{i-2})(\lambda_{\alpha}-(1+\alpha)k)\\&~-(1-\alpha)(t_1^{i-3}-t_2^{i-3})k]+(1-\alpha)(x_1-x_0)[(t_1^{i-1}-t_2^{i-1})-(t_1^{i-2}-t_2^{i-2})]\\
\geq&~[(t_1^{i-1}-t_2^{i-1})(\lambda_{\alpha}-(1+\alpha)k)-(1-\alpha)(t_1^{i-2}-t_2^{i-2})k]- [(t_1^{i-2}-t_2^{i-2})(\lambda_{\alpha}-(1+\alpha)k)-(1-\alpha)(t_1^{i-3}-t_2^{i-3})k]\\
\geq&~[(t_1^{i-2}-t_2^{i-2})t_1(\lambda_{\alpha}-(1+\alpha)k)-(t_1^{i-2}-t_2^{i-2})k]- (t_1^{i-2}-t_2^{i-2})(\lambda_{\alpha}-(1+\alpha)k)\\
>&~(t_1^{i-2}-t_2^{i-2})(\lambda_{\alpha}-(1+\alpha)k)(t_1-2) \\
>&~0.
\end{align*}
\end{comment}

Let $h_1(t_1,t_2)=(t_1^{i-1}-t_2^{i-1})\big[(\lambda_{\alpha}-(1+\alpha)k)x_1+(1-\alpha)(x_1-x_0)\big] -k(1-\alpha)(t_1^{i-2}-t_2^{i-2})x_1$, and
 $h_2(t_1,t_2)=(t_1^{i-2}-t_2^{i-2})\big[(\lambda_{\alpha}-(1+\alpha)k)x_1+(1-\alpha)(x_1-x_0)\big] -k(1-\alpha)(t_1^{i-3}-t_2^{i-3})x_1$.

Notice that $x_1>x_0$, then we have
\begin{align*}
h_1(t_1,t_2)-h_2(t_1,t_2)=&\big[(t_1^{i-1}-t_2^{i-1})-(t_1^{i-2}-t_2^{i-2})\big](\lambda_{\alpha}-(1+\alpha)k)x_1\\&+\big[(t_1^{i-3}-t_2^{i-3})-(t_1^{i-2}-t_2^{i-2})\big](1-\alpha)kx_1\\&+\big[(t_1^{i-1}-t_2^{i-1})-(t_1^{i-2}-t_2^{i-2})\big](1-\alpha)(x_1-x_0)\\>&\big[(t_1^{i-1}-t_2^{i-1})-(t_1^{i-2}-t_2^{i-2})\big](\lambda_{\alpha}-(1+\alpha)k)x_1\\&+\big[(t_1^{i-3}-t_2^{i-3})-(t_1^{i-2}-t_2^{i-2})\big](1-\alpha)kx_1\\>&(t_1^{i-1}-t_2^{i-1})-2(t_1^{i-2}-t_2^{i-2})+(t_1^{i-3}-t_2^{i-3})\big](1-\alpha)kx_1\\>&(t_1^{i-3}-t_2^{i-3})(1-\alpha)kx_1>0.
\end{align*}
Thus we have
$x_{i}>x_{i-1}$, for each $i=2,\ldots,l-1$ when $l\geq 4$. By symmetry, we have
$x_{j}>x_{j+1}$, for each $l+1\leq j\leq d-2$ where $d-l\geq 4$.

\textbf{Case 2.} $l=1,2,3$.

 It is clear that $x_{l-1}>\cdots>x_1>x_0$ holds for $l=1$ and $l=2$. Next we are going to proof that $x_{l-1}>\cdots>x_1>x_0$ holds for $l=3$, that is $x_2>x_1>x_0$. By the eigenequations of $A_\alpha(G)$ corresponding to $x_0$ and $x_1$, we have
\begin{equation*}
\left\{
\begin{array}{c}
\begin{aligned}
&(1-\alpha) \frac{kx_{1}}{x_{0}}+( k \alpha-\lambda_\alpha)=0, \\
&(1-\alpha) \frac{kx_{2}}{x_{1}}+[(k+1) \alpha+k-1-\lambda_\alpha] +(1-\alpha) \frac{x_{0}}{x_{1}}=0.
\end{aligned}
\end{array}
\right.
\end{equation*}
%\begin{comment}
Then we have
\begin{equation*}
\left\{
\begin{array}{c}
\begin{aligned}
&\frac{x_{1}}{x_{0}}=\frac{\lambda_\alpha- k \alpha}{(1-\alpha)k}, \\
&\frac{x_{2}}{x_{1}}=\frac{\lambda_{\alpha}+1-k-(k+1)a}{(1-\alpha)k}-\frac{x_0}{kx_1}. \\
\end{aligned}
\end{array}
\right.
\end{equation*}
%\end{comment}
Thus $\frac{x_{2}}{x_{1}}=\frac{(\lambda_{\alpha}+1-k-(k+1)\alpha)(\lambda_{\alpha}-k\alpha)-(1-\alpha)^2k}{(\lambda_{\alpha}-k\alpha)(1-\alpha)k}.$ Let $g_1=(\lambda_{\alpha}+1-k-(k+1)\alpha)(\lambda_{\alpha}-k\alpha)-(1-\alpha)^2k$, $g_2=(\lambda_{\alpha}-k\alpha)(1-\alpha)k$. Therefore, in order to prove $\frac{x_2}{x_1}>1$, we only need to prove $g_1-g_2>0$. Let $g_3=g_1-g_2$. By direct calculation, we have
\begin{align*}
g_3=g_1-g_2
%&=(\lambda_{\alpha}+1-k-(k+1)\alpha-k+k\alpha)(\lambda_{\alpha}-k\alpha)-(1-\alpha)^2k\\
%&=(\lambda_{\alpha}-2k+1-\alpha)(\lambda_{\alpha}-k\alpha)-(1-\alpha)^2k\\
=\lambda_{\alpha}^2+(1-k\alpha-2k-\alpha)\lambda_{\alpha}-k(1-2k\alpha-\alpha).
\end{align*}
Note that $g_3$ is a fuction of $\lambda_{\alpha}$, so $g_3 = 0$ has solutions if and only if  $(1-k\alpha-2k-\alpha)^2+4k(1-2k\alpha-\alpha)>0$.
Note that  $(1-k\alpha-2k-\alpha)^2+4k(1-2k\alpha-\alpha)=(k+1)^2\alpha^2-(4k^2+2k+2)\alpha+4k^2+1$. By direct calculation, it is easy to prove that $(k+1)^2\alpha^2-(4k^2+2k+2)\alpha+4k^2+1>0$ holds  for $\alpha\in[0,1)$ and $k\geq 2$.

Thus in order to prove $g_3=g_1-g_2>0$, we need to prove that $\lambda_{\alpha}$ is greater than the maximum solution of $g_3=0$, that is $\lambda_{\alpha}>\frac{1}{2}(\alpha+\alpha k+2k-1+\sqrt{(\alpha-2)^2k^2+2\alpha(\alpha-1)k+(\alpha-1)^2})$. Notice that $\lambda_{\alpha}>3k-1$, then we need to prove that $3k-1>\frac{1}{2}(\alpha+\alpha k+2k-1+\sqrt{(\alpha-2)^2k^2+2\alpha(\alpha-1)k+(\alpha-1)^2})$, which is equivalent to prove $2(3k-1)-\alpha-\alpha k-2k+1-[(\alpha-2)^2k^2+2\alpha(\alpha-1)k+(\alpha-1)^2]=(12-4\alpha)k^2-(4\alpha+8)k+4\alpha>0$. Let $g_4=(12-4\alpha)k^2-(4\alpha+8)k+4\alpha$. By direct calculate, we have $(4\alpha+8)^2-16\alpha(12-4\alpha)=16(5\alpha^2-8\alpha+4)>0$ holds for each $\alpha\in[0,1)$. Thus, if k is greater than the maximum solution of $g_4=0$, then $g_4>0$. By direct calculation, we have the maximum solution of $g_4=0$ is $\frac{\alpha+\sqrt{(5\alpha^2-8\alpha+4}+2}{2(3-\alpha)}\in[0.6667,1)$ for each $\alpha\in[0,1)$. Thus $g_4>0$ holds for each $k\geq 2$ and each $\alpha\in[0,1)$. Thus $g_3=g_1-g_2>0$. Then  $\frac{x_2}{x_1}>1$.

Thus we have
$x_{i}>x_{i-1}$, for each $i=2,\ldots,l-1$ when $1\leq l\leq 3$. By symmetry, we have
$x_{j}>x_{j+1}$, for each $l+1\leq j\leq d-2$ when $1\leq d-l\leq 3$.

These complete the proof of (i) in this Lamma \ref{th-2}. Next we are going to prove $n_lx_l>kx_{l-1}$.

\textbf{(ii) \bm{$n_lx_l>kx_{l-1}$}. }

By the eigenequation corresponding to $x_{l-1}$, that is
$$(1-\alpha) \frac{n_lx_l}{x_{l-1}}+[(n_l+k) \alpha+k-1-\lambda_\alpha] +(1-\alpha)\frac{kx_{l-2}}{x_{l-1}}=0.$$
Note that $k\geq 2$, $x_{l-1}>x_{l-2}$ and $\alpha\in[0,1)$. Then we have
\begin{align*}
\frac{n_lx_l}{kx_{l-1}}&=\frac{1}{k(1-\alpha)}[-(n_l+k) \alpha-k+1+\lambda_\alpha] -\frac{x_{l-2}}{x_{l-1}},\\
&>\frac{1}{k(1-\alpha)}[-(n_l+k) \alpha-k+1+\lambda_\alpha]-1,\\
&>\frac{1}{k(1-\alpha)}[-(n_l+k) \alpha-k+1+(n_l+k-1)]-1>1.
\end{align*}
Thus we have $n_lx_l>kx_{l-1}$. These complete the proof of the first and second results (i) and (ii) in this Lamma \ref{th-2}.

\textbf(iii) \textbf{$\bm{x_d>x_0}$ and $\bm{x_i<x_{d-i}}$ for each $\bm{i=1,2,\ldots,d-l-1}$, where $\bm{l\geq d-l+2}$.}

By the eigenequations of $A_\alpha(G)$, we have
\begin{equation*}
\left\{
\begin{aligned}
\begin{array}{l}
(1-\alpha) \frac{kx_{1}}{x_{0}}+( k \alpha-\lambda_\alpha)=0, \\
(1-\alpha) \frac{kx_{2}}{x_{1}}+[(k+1) \alpha+k-1-\lambda_\alpha] +(1-\alpha) \frac{x_{0}}{x_{1}}=0, \\
(1-\alpha) \frac{kx_{3}}{x_{2}}+(2 k \alpha+k-1-\lambda_\alpha) +(1-\alpha) \frac{kx_{1}}{x_{2}}=0,\\
\vdots\\
(1-\alpha) \frac{kx_{d-l-1}}{x_{d-l-2}}+(2 k \alpha+k-1-\lambda_\alpha) +(1-\alpha) \frac{kx_{d-l-3}}{x_{d-l-2}}=0,\\
(1-\alpha) \frac{kx_{d-l}}{x_{d-l-1}}+(2 k \alpha+k-1-\lambda_\alpha) +(1-\alpha) \frac{kx_{d-l-2}}{x_{d-l-1}}=0,\\
\vdots\\
(1-\alpha) \frac{kx_{l-1}}{x_{l-2}}+(2 k \alpha+k-1-\lambda_\alpha) +(1-\alpha) \frac{kx_{l-3}}{x_{l-2}}=0,\\
(1-\alpha) \frac{n_lx_l}{x_{l-1}}+[(n_l+k) \alpha+k-1-\lambda_\alpha] +(1-\alpha)\frac{kx_{l-2}}{x_{l-1}}=0,\\
(1-\alpha)\frac{kx_{l-1}}{x_l}+(2k\alpha+n_l-1-\lambda_\alpha) +(1-\alpha)\frac{kx_{l+1}}{x_l
}=0,\\
(1-\alpha) \frac{n_lx_l}{x_{l+1}}+[(n_l+k) \alpha+k-1-\lambda_\alpha] +(1-\alpha)\frac{kx_{l+2}}{x_{l+1}}=0,\\
(1-\alpha) \frac{kx_{l+1}}{x_{l+2}}+(2 k \alpha+k-1-\lambda_\alpha) +(1-\alpha) \frac{kx_{l+3}}{x_{l+2}}=0,\\
\end{array}
\end{aligned}
\right.
\end{equation*}
\begin{equation*}
\left\{
\begin{aligned}
\begin{array}{l}
\vdots\\
(1-\alpha) \frac{kx_{d-3}}{x_{d-2}}+(2 k \alpha+k-1-\lambda_\alpha) +(1-\alpha) \frac{kx_{d-1}}{x_{d-2}}=0,\\
(1-\alpha) \frac{kx_{d-2}}{x_{d-1}}+[(k+1) \alpha+k-1-\lambda_\alpha] +(1-\alpha) \frac{x_{d}}{x_{d-1}}=0, \\
(1-\alpha) \frac{kx_{d-1}}{x_{d}}+( k \alpha-\lambda_\alpha)=0. \\
\end{array}
\end{aligned}
\right.
\end{equation*}
Since $l\geq d-l+1$ and $k\geq 2$, then
\begin{equation}\label{x/y}
\frac{kx_{d-1}}{x_{d}}=\frac{kx_{1}}{x_{0}},\ \frac{kx_{d-2}}{x_{d-1}}=\frac{kx_{2}}{x_{1}},\ldots,\ \frac{kx_{l+1}}{x_{l+2}}=\frac{kx_{d-l-1}}{x_{d-l-2}}.
\end{equation}
Notice that $n_l\geq 2k$, by the eigenequations corresponding to $x_{d-l-1}$, $x_{l+1}$, these are
\begin{equation*}
\left\{
\begin{aligned}
\begin{array}{l}
(1-\alpha)\frac{kx_{d-l}}{x_{d-l-1}}+(2 k \alpha+k-1-\lambda_\alpha) +(1-\alpha) \frac{kx_{d-l-2}}{x_{d-l-1}}=0,\\
(1-\alpha) \frac{n_lx_l}{x_{l+1}}+[(n_l+k) \alpha+k-1-\lambda_\alpha] +(1-\alpha)\frac{kx_{l+2}}{x_{l+1}}=0.\\
\end{array}
\end{aligned}
\right.
\end{equation*}
Thus $\frac{kx_{d-l}}{x_{d-l-1}}>\frac{n_lx_l}{x_{l+1}}$. Since $n_lx_l>kx_{d-l}$, then $x_{d-l-1}<x_{l+1}$, by Equation (\ref{x/y}), we have $x_i<x_{d-i}$ and $x_d>x_0$, for each $i=1,2,\ldots,d-l-1$.

%If $G$ is the maximal graph in $\mathcal{G}_{n,k}^k$. Let $G'$ be a graph obtained from $G$ by deleting all edges between $Z_l$ and $V_{l+1}$ and adding all edges between $Z_l$ and $V_{l-2}$. If $x_{l-2}\geq x_{l+1}$, then by Lemma \ref{v-u}, we have $\lambda_{\alpha}(G')>\lambda_{\alpha}(G)$. This contradicts that $G$ is the maximal graph in $\mathcal{G}_{n,k}^k$. Thus $x_{l-2}\leq x_{l+1}$.  Similarly, we have $x_{l-1}>x_{l+2}$.

These complete the proof of this Lemma \ref{th-2}.
\end{proof}
In the following, we use all the notations in Lemma \ref{th-2} unless otherwise stated.
\begin{lemma}\label{th-3}
Let $G=G(l,d-l)$ be a maximal graph in $\mathcal{G}_{n,k}^d$. %as shown in Figure \ref{fig1}.
Suppose $l\geq d-l+2$, then $x_{d-i+1}<x_i,\ i=1,2,\ldots,d-l-1$.
\end{lemma}

\begin{proof}
 By Lemma \ref{th-2}, we have $x_0<x_d$, so $x_{d-i}\neq x_i$, $i=1,2,\ldots,d-l-1$.  Recall that $Z=Z_l=\{V_l\setminus V(K_k)\}$. Let $G'$ be a graph obtained from the maximal graph $G$ by deleting all edges between $Z$ and $V_{l+1}$ and adding all edges between $Z$ and $V_{l-2}$. If $x_{l-2}\geq x_{l+1}$, then by Lemma \ref{v-u}, we have $\lambda_{\alpha}(G')>\lambda_{\alpha}(G)$. Thus this contradicts that $G$ is the maximal graph in $\mathcal{G}_{n,k}^k$. Thus $x_{l-2}<x_{l+1}$.  Similarly, we have $x_{l-1}>x_{l+2}$. Next we prove $x_{d}<x_1$.

Suppose $x_{d}>x_1$. We select one vertex $w$ of $V_1$. Since $n_0=n_d=1$, then we denote the two vertices in $V_0$, $V_d$ as $u$, $v$, respectively. Let $G''$ be a graph obtained from $G$ by deleting  all edges between $V_{1}\setminus \{w\}$ and $w$, all edges between $V_{1}\setminus \{w\}$ and $V_{2}$, edge $wu$, and adding edge $vu$, all edges between $V_{1}\setminus\{w\}$ and $V_{d-1}$, all edges between $V_{1}\setminus\{w\}$ and $v$. Obviously, $G''\in \mathcal{G}_{n,k}^d$. Then
\begin{comment}
\begin{align*}
0&>\lambda_{\alpha}(G'')-\lambda_{\alpha}(G)\\&>(k-1)(x_d-x_1)[(2-\alpha)x_1+\alpha x_d]+(k-1)(x_{d-1}-x_2)[2(1-\alpha)x_1+\alpha(x_{d-1}+x_2)].
\end{align*}
By Lemma \ref{th-2}, we have $x_0<x_1<x_2$, $x_1<x_{d-1}$ and $x_d<x_{d-1}$. Then
\end{comment}
\begin{align*}
0>&\lambda_{\alpha}(G'')-\lambda_{\alpha}(G)\\>&(1-\alpha)\big[k(k-1)x_1(x_{d-1}-x_2)+(k-1)x_1(x_d-x_1)+x_0(x_d-x_1)\big]\\+&\alpha\big[(k-1)(x_{d-1}^2-x_2^2)+k(x_d^2-x_1^2)\big]\\>&(1-\alpha)k(k-1)x_1(x_{d-1}-x_2)+\alpha(k-1)(x_{d-1}^2-x_2^2).
\end{align*}
Thus if $x_d>x_1$, then $x_{d-1}<x_2$ holds.

Next we obtain the graph $G''$ by using a different edge-shifting operation. Let $G''$ be the graph obtained from $G$ by deleting all edges between $V_{d-2}$ and $V_{d-1}$, all edges between $V_{2}$ and $V_{3}$ and adding all edges between $V_{d-2}$ and $V_{2}$, all edges between $V_{d-1}$ and $V_{3}$. Then we have
\begin{align*}
0&>\lambda_{\alpha}(G'')-\lambda_{\alpha}(G)>(1-\alpha)k^2(x_{d-2}-x_{3})(x_{2}-x_{d-1}).
\end{align*}
Thus if $x_d>x_1$ and $x_{d-1}<x_2$, then $x_{d-2}<x_3$ holds. Through gradual recursion, let $G''$ be the graph obtained from $G$ by deleting all edges between $V_{j}$ and $V_{j+1}$, all edges between $V_{d-j}$ and $V_{d-j+1}$ and adding all edges between $V_{j}$ and $V_{d-j}$, all edges between $V_{j+1}$ and $V_{d-j+1}$, for $j=2,\ldots,l-2$. Thus we have if $x_d>x_1$,   then $x_{d-1}<x_2$, $x_{d-2}<x_3$,$\ldots$, $x_{l+1}<x_{d-l}<x_{d-l+1}<\ldots<x_{l-2}$.
There is a contradiction that $x_{l-2}<x_{l+1}$ in the maximal graph $G$. Thus $x_{d}<x_1$.

Assume
$x_{d-i+1}<x_{i}$ for $i=1,2,\ldots,d-l-2$,
and we now prove $x_{d-i}<x_{i+1}$.
Let $G''$ be a graph obtained from $G$ by deleting all edges between $V_{i}$ and $V_{i+1}$, all edges between $V_{d-i}$ and $V_{d-i+1}$ and adding all edges between $V_{i}$ and $V_{d-i}$, all edges between $V_{i+1}$ and $V_{d-i+1}$. Then
$$0>\lambda_{\alpha}(G'')-\lambda_{\alpha}(G)>(1-\alpha)k^2(x_{d-i}-x_{i+1})(x_{i}-x_{d-i+1}).$$
Since $x_{i}>x_{d-i+1}$,
we have $x_{d-i}<x_{i+1}$, for each $i=1,2,\ldots,d-l-2$.
These complete the proof.
\end{proof}
Now we have all the ingredients to present our proof of Theorem \ref{th-main}.

\begin{proof}[\textbf{Proof of Theorem \ref{th-main}}]
In this proof, we use the notations in Lemma \ref{th-2} unless otherwise stated.
Let $G^*=G(l,d-l)=\overbrace{K_{n_0}\vee\cdots\vee K_{n_{l-1}}}^{l}\vee K_{n_l}\vee\overbrace{K_{n_{l+1}}\cdots\vee K_{n_d}}^{d-l}$, where $n_0=n_d=1$, $n_{l}\geq 2k$, $n_i=k$,  $i\in\{1,2,\ldots,d-1\}\setminus\{l\}$, as shown in Figure \ref{fig1}.
 By Lemma \ref{th-1}, the result hold when $d=2$ or $3$. Thus we only consider the case $d\geq 4$ in the following.

 In order to prove a contradiction, we suppose $l\geq d-l+2$. Let $G''$ be a graph obtained from $G^*$ by deleting all the edges between $V_l$ and $V_{l+1}$ and
 the edges between $V_{d-l+1}$ and $V_{d-l+2}$, then adding all the edges between  $V_{l+1}$ and $V_{d-l+2}$ and
 the edges between $V_l$ and $V_{d-l+1}$. Evidently,
 $G''=G(l-1,d-l+1)\in\mathcal{G}_{n,k}^d$. Note that $\alpha n_lx_l > kx_{l-1}>kx_{d-l+2}$ by Lemma \ref{th-2} and $x_{l+1}>x_{l-2}>x_{l-1}>x_{d-l+1}$. Then from Equation \ref{xax}, we have
 \begin{align*}
\lambda_{\alpha}(G'')-\lambda_{\alpha}(G^*)\geq&(1-\alpha)[-n_ln_{l+1}x_lx_{l+1}-n_{d-l+1}n_{d-l+2}x_{d-l+1}x_{d-l+2}+\\&n_{l+1}n_{d-l+2}x_{l+1}x_{d-l+2}+n_{l}n_{d-l+1}x_{l}x_{d-l+1}]+\\&\alpha[(n_{l}-n_{d-l+1})x_{d-l+1}^2+(n_{l+1}-n_{d-l+1})x_{d-l+2}^2+(n_{l}-n_{d-l+2})x_{l+1}^2]\\
\geq&(1-\alpha)(n_{d-l+1}x_{d-l+1}-n_{l+1}x_{l+1})(n_lx_l-n_{d-l+2}x_{d-l+2})+\\&\alpha(2k+1-2k)x_{d-l+1}^2\geq 0.
 \end{align*}
 %Let $G''$ be the graph obtained from $G$ by deleting all edges between $V_{l+1}$ and $V_{l+2}$ and all edges between $V_{d-l-1}$ and $V_{d-l}$, and adding all edges between $V_{l+2}$ and $V_{d-l}$ and all edges between $V_{d-l}$ and $V_{l+1}$.  By Lemmas \ref{th-1}, \ref{th-2} and \ref{th-3}, we have $x_{l+2}<x_{d-l-1}$, $x_{l+1}>x_{d-l}$.

%Thus
%\begin{align*}
%\lambda_{\alpha}(G'')-\lambda_{\alpha}(G)\geq &2(1-\alpha)k^2\big[(x_{d-l-1}-x_{l+2})x_{l+1}+(x_{l+1}-x_{d-1})x_{d-l-1}\big]>0.
%\end{align*}
Thus $\lambda_{\alpha}(G'')>\lambda_{\alpha}(G^*)$, which contradicts the fact that $G^*$ is the maximal graph in $\mathcal{G}_{n,k}^k$. Thus we have $l\leq d-l+1$.

 Notice that $\Delta(G^*)=(n_{\lfloor\frac{d}{2}\rfloor}+2k-1)$, by Lemmas \ref{le:bound1} and \ref{le:bound2}, our conclusion can be obtained directly.
These complete the proof.
\end{proof}

\end{document}